\documentclass[11pt]{article}
\usepackage[latin1]{inputenc}
\usepackage{amsmath,amsthm,amssymb}
\usepackage{amsfonts}
\usepackage{amsmath,amsthm,amssymb,amscd}
\usepackage{latexsym}
\usepackage{color}
\usepackage{graphicx}
\usepackage{mathrsfs}
\usepackage{cite}

\usepackage{color,enumitem,graphicx}
\usepackage[colorlinks=true,urlcolor=black,
citecolor=black,linkcolor=black,linktocpage,pdfpagelabels,
bookmarksnumbered,bookmarksopen]{hyperref}

\makeatletter \@addtoreset{equation}{section} \makeatother

\newtheorem{theorem}{Theorem}[section]
\newtheorem{definition}{Definition}[section]
\newtheorem{proposition}{Proposition}[section]
\newtheorem{lemma}{Lemma}[section]
\newtheorem{remark}{Remark}[section]

\newtheorem{corollary}[theorem]{Corollary}

\allowdisplaybreaks

\begin{document}
\title{Remainder terms and sharp
quantitative stability for a nonlocal Sobolev inequality on the Heisenberg group}

\author{Wenjing Chen\footnote{Corresponding author.} \footnote{E-mail address:\, {\tt wjchen@swu.edu.cn} (W. Chen), {\tt zxwangmath@163.com} (Z. Wang).}\  \ and Zexi Wang\\
\footnotesize  School of Mathematics and Statistics, Southwest University,
Chongqing, 400715, P.R. China}


\date{ }
\maketitle

\begin{abstract}
{
In this paper, we study the following nonlocal Sobolev inequality on the Heisenberg group
\begin{equation}\label{eq:HLS}
	S_{HL}(Q,\mu)	\left(\int_{\mathbb{H}^{n}}\int_{\mathbb{H}^{n}}\frac{|u(\xi)|^{Q^{\ast}_{\mu}}|u(\eta)|^{Q^{\ast}_{\mu}}}{|\eta^{-1}\xi|^{\mu}}{d}\xi{d}\eta\right)^{\frac{1}{Q^{\ast}_{\mu}}}\leq \int_{\mathbb{H}^{n}}|\nabla_{\mathbb{H}}u|^{2}d\xi,\quad \forall \,  u\in S^{1,2}(\mathbb{H}^{n}),
\end{equation}
where $Q=2n+2$ is the homogeneous dimension of the Heisenberg group $\mathbb{H}^{n}$, $n\geq1$, $\mu\in(0,Q)$, $Q^{\ast}_{\mu}=\frac{2Q-\mu}{Q-2}$ is the upper critical exponent in the sense of  the Hardy-Littlewood-Sobolev inequality and the Folland-Stein-Sobolev inequality on the Heisenberg group, $S_{HL}(Q,\mu)$ is  the sharp constant of \eqref{eq:HLS}, and $S^{1,2}(\mathbb{H}^{n})$ is the Folland-Stein-Sobolev space. 
It is well-known that, up to a translation and suitable scaling,
\begin{equation}\label{eq:abs}
  -\Delta_{\mathbb{H}} u=\left(\int_{\mathbb{H}^{n}}\frac{|u(\eta)|
^{Q^{\ast}_{\mu}}}{|\eta^{-1}\xi|^{\mu}}{d}\eta\right)|u|^{Q_\mu^*-2}u,~~u\in S^{1,2}(\mathbb{H}^{n})
\end{equation}
is the Euler-Lagrange equation corresponding to the associated minimization problem.

On the one hand, we show the existence of a gradient-type
remainder term for inequality \eqref{eq:HLS} when $Q\geq4$, $\mu\in (0,4]$, and  as a corollary, derive the existence of a remainder term in the weak $L^{\frac{Q}{Q-2}}$-norm on bounded domains. On the other hand, we establish the quantitative stability of critical points for equation \eqref{eq:abs} in the multi-bubble case when $Q=4$ and $\mu\in (2,4)$.
 }



\end{abstract}

\section{Introduction}
\subsection{Stability of Sobolev inequality in $\mathbb{R}^{N}$}
The classical Sobolev inequality
states that for $N\geq3$, there exists a dimensional constant $\mathcal{S}=\mathcal{S}(N)>0$ such that
\begin{equation}\label{Sobolev inequality}
	\|\nabla u\|_{L^{2}(\mathbb{R}^{N})}^{2}\geq \mathcal{S}\|u\|_{L^{2^{\ast}}(\mathbb{R}^{N})}^{2},\quad \forall \, u\in{D}^{1,2}(\mathbb{R}^{N}),
	\end{equation}
where $2^{\ast}=\frac{2N}{N-2}$ denotes the critical exponent for the Sobolev embedding $D^{1,2}(\mathbb{R}^{N})\hookrightarrow L^{p}(\mathbb{R}^{N})$, and $D^{1,2}(\mathbb{R}^N)$ is defined as the completion of $C_c^\infty(\mathbb{R}^{N})$ with
respect to the norm
\begin{equation*}
\|u\|_{D^{1,2}(\mathbb{R}^{N})}=\left(\int_{\mathbb{R}^{N}}|\nabla u|^{2}{d}x\right)^{\frac{1}{2}}.
\end{equation*}
In \cite{T76} (see also Aubin \cite{A76}), Talenti computed the sharp constant of \eqref{Sobolev inequality} and showed that its extremal functions are exactly the Talenti bubbles, which take the form
\begin{equation*}
U_{\lambda,a}(x)=[N(N-2)]^{\frac{N-2}{4}}\left(\frac{\lambda}{1+\lambda^{2}|x-a|^{2}}\right)
^{\frac{N-2}{2}},~~\lambda>0,~a\in\mathbb{R}^{N}.
\end{equation*}
Furthermore, Chen and Li \cite{CL91} established that $U_{\lambda,a}$ are the unique
positive solutions to the Euler-Lagrange equation corresponding to \eqref{Sobolev inequality}
\begin{equation}\label{limit1'}
  -\Delta u=|u|^{2^*-2}u,\qquad \mathrm{in}~\mathbb{R}^{N}.
\end{equation}

In \cite{BL85}, Brezis and Lieb posed a fundamental question: For any \( u \in D^{1,2}(\mathbb{R}^{N}) \), can one naturally bound the quantity \( \|\nabla u\|_{L^{2}(\mathbb{R}^{N})}^{2} - \mathcal{S}\|u\|_{L^{2^{\ast}}(\mathbb{R}^{N})}^{2} \) from below by the ``distance" between \( u \) and the manifold of bubble functions \( \mathcal{M} \)? Here $\mathcal{M}=\big\{cU_{\lambda,a}: c\in \mathbb{R},\lambda>0,a\in \mathbb{R}^N\big\}$ is an $(N+2)$-dimensional manifold.

This question
was first addressed by Bianchi and Egnell \cite{BE91}. To analyze the remainder term  in the Sobolev inequality \eqref{Sobolev inequality}, they first established the non-degeneracy of
 $U_{1,0}$, then employed the global-local analysis and spectral theory to prove that
 there exists  $\kappa>0$ such that
\begin{equation*}
  \|\nabla u\|_{L^{2}(\mathbb{R}^{N})}^{2}-\mathcal{S}\|u\|_{L^{2^{\ast}}(\mathbb{R}^{N})}^{2}\geq \kappa ~  \mathrm{dist}(u,\mathcal{M})^2,
  \quad \forall \, u\in{D}^{1,2}(\mathbb{R}^{N}),
\end{equation*}
where  $\mathrm{dist}(u,\mathcal{M})=\inf_{c\in \mathbb{R},\lambda>0,a\in \mathbb{R}^N}\|u-cU_{\lambda,a}\|_{D^{1,2}(\mathbb{R}^{N})}$.
Subsequent works extended this result to the second Sobolev inequality \cite{LW2000},
the fractional Sobolev setting (for all \( N \geq 1 \) and \( s \in (0,\frac{N}{2}) \); \cite{CFW13}), and the \( p \)-Laplacian case \cite{FN19,FZ22,N20}.
For more related results, including explicit lower or upper bounds for \( \kappa \), we refer the readers to \cite{DEFF25,K23,K24,K25,CLT25,CLT24,CLT24'}.


A natural and more challenging direction is to investigate the quantitative stability of critical points for
the Euler-Lagrange equation \eqref{limit1'}: Specifically, if a function almost solves \eqref{limit1'}, must it be quantitatively close to a single bubble \( U_{\lambda,a} \) or a sum of weakly interacting bubbles? Struwe's celebrated global compactness lemma \cite{S84} addressed this question qualitatively. Building on \cite{S84}, extensive research has been dedicated to the stability of critical points for \eqref{limit1'}. A significant breakthrough was achieved by Ciraolo, Figalli, and Maggi \cite{CFM18}, who established a sharp quantitative stability result for the single-bubble case in dimensions \( N \geq 3 \). Precisely, they showed that
\begin{equation*}
  \|u-U_{\lambda,a}\|_{{D}^{1,2}(\mathbb{R}^{N})} \leq  C
 \Theta(u).
\end{equation*}
where $\Theta(u)=\|\Delta u+|u|^{2^*-2}u\|_{({D}^{1,2}(\mathbb{R}^{N}))^{-1}}$.
Subsequently, Figalli and Glaudo \cite{FG20} and Deng, Sun and Wei \cite{DSW25} extended this
stability result to the multi-bubble (weakly interacting) case for $3\leq N \leq5$ and $N\geq6$, respectively. Their results can be summarized as follows
\begin{equation*}
\left\|u-\sum_{i=1}^{m}{U}_{\lambda_i, a_i} \right\|_{{D}^{1,2}(\mathbb{R}^{N})} \leq C
\begin{cases}
 \Theta(u), \qquad & \mathrm{if} ~ 3\leq N\leq 5, \\
  \Theta(u)|\log \Theta(u) |^{\frac{1}{2}}, \qquad & \mathrm{if} ~ N=6, \\
  \Theta(u)^{\frac{N+2}{2(N-2)}}, \qquad & \mathrm{if} ~ N\geq 7.
\end{cases}
\end{equation*}
  Furthermore, the quantitative stability results in \cite{CFM18,DSW25,FG20}
have been generalized to the fractional setting \cite{CKW25} and the $p$-Laplacian case  (single-bubble case; \cite{CG25,LZ25}).
We now turn to the nonlocal inequalities associated with the Hardy-Littlewood-Sobolev (HLS) inequality, which we recall below.
\begin{proposition} \cite[Theorem 4.3]{LL01}
Suppose that  $N\geq 1$, $0<\mu< N$ and $t,r>1$ with $\frac{1}{t}+\frac{\mu}{N}+\frac{1}{r}=2$, $f\in L^{t}(\mathbb{R}^{N})$ and $h\in L^{r}(\mathbb{R}^{N})$. Then there exists a constant ${C}(N,\mu,t,r)>0$ independent of $f$ and $h$ such that
\begin{equation}\label{HLS}
  \displaystyle\int_{\mathbb{R}^N}\int_{\mathbb{R}^N}\frac{f(x)h(y)}{|x-y|^\mu}dxdy\leq C(N,\mu,t,r)\|f\|_{L^t(\mathbb{R}^N)}\|h\|_{L^r(\mathbb{R}^N)}.
\end{equation}
If $t=r=\frac{2N}{2N-\mu}$, then
\begin{equation*}
  {C}(N,\mu,t,r)=C(N,\mu)=\frac{\pi^{\frac{\mu}{2}}\Gamma\big(\frac{N-\mu}{2}\big)}{\Gamma\big(
  \frac{2N-\mu}{2}\big)}\left(\frac{\Gamma(N)}{\Gamma(\frac{N}{2})}\right)^{\frac{N-\mu}{N}},
\end{equation*}
and there is equality in \eqref{HLS} if and only if $f\equiv (const.)h$ and
\begin{equation*}
  h(x)=c\left(\frac{\lambda}{1+\lambda^{2}|x-a|^{2}}\right)
^{\frac{2N-\mu}{2}},~~\lambda>0,~a\in\mathbb{R}^{N}
\end{equation*}
for some $c\in \mathbb{R}$, 
where
$\Gamma$ denotes the Gamma function $\Gamma(\gamma)=\int_0^{+\infty}t^{\gamma-1}e^{-t}dt$ for $\gamma>0$.
\end{proposition}
According to \eqref{HLS}, the functional
\begin{equation*}
  \int_{\mathbb{R}^{N}}\int_{\mathbb{R}^{N}} \frac{|u(x)|^p|u(y)|^p}{|x-y|^\mu}dxdy
\end{equation*}
is well defined in $D^{1,2}(\mathbb{R}^{N})$ if $p\in\big[\frac{2N-\mu}{N},\frac{2N-\mu}{N-2}\big]$. Then $2^{\ast}_{\mu}=\frac{2N-\mu}{N-2}$ and  $2_{\ast}^{\mu}=\frac{2N-\mu}{N}$ are called the upper and lower critical exponents, respectively.
Moreover, combining the Sobolev inequality with the HLS
inequality, one can easily derive the following nonlocal inequality
\begin{equation}\label{NonS}
	\mathcal{S}_{HL}(N,\mu)	\left(
\int_{\mathbb{R}^{N}}\int_{\mathbb{R}^{N}}
\frac{|u(x)|^{2_\mu^*}|u(y)|^{2_\mu^*}}{|x-y|^\mu}dxdy
\right)^{\frac{1}{2^{\ast}_{\mu}}}\leq \int_{\mathbb{R}^{N}}|\nabla u|^{2}d x,\quad \forall \,  u\in D^{1,2}(\mathbb{R}^{N}),
\end{equation}
where $\mathcal{S}_{HL}(N,\mu)=\frac{\mathcal{S}}{C(N,\mu)^{\frac{1}{2_\mu^*}}}$.
And the associated Euler-Lagrange equation becomes
\begin{equation}\label{limit2'}
  -\Delta u=\left(\int_{\mathbb{R}^{N}}\frac{|u(y)|
^{2^{\ast}_{\mu}}}{|x-y|^{\mu}}dy \right)|u|^{2_\mu^*-2}u,~~u\in D^{1,2}(\mathbb{R}^{N}).
\end{equation}
To gain a deeper understanding of inequality \eqref{NonS}, it is natural to investigate the quantitative stability of both \eqref{NonS} and its Euler-Lagrange equation \eqref{limit2'}.
Inspired by \cite{BE91}, using the classification \cite{DY19,GHPS19} and non-degeneracy \cite{LLTX23} of positive solutions
 to \eqref{limit2'},
 Deng et al. \cite{DTYZ23} established a gradient-type remainder term estimate for \eqref{NonS}: Assume that $N\geq3$, $0<\mu<N$ with $\mu\leq 4$, then there exist constants $\kappa_2>\kappa_1>0$ such that for all $u\in{D}^{1,2}(\mathbb{R}^{N})$, it holds that
\begin{equation*}
\kappa_2 ~  \mathrm{dist}(u,\mathcal{M})^2\geq  \int_{\mathbb{R}^{N}}|\nabla u|^{2}{d}x-\mathcal{S}_{HL}(N,\mu)\left(
\int_{\mathbb{R}^{N}}\int_{\mathbb{R}^{N}}
\frac{|u(x)|^{2_\mu^*}|u(y)|^{2_\mu^*}}{|x-y|^\mu}dxdy
\right)^{\frac{1}{2^{\ast}_{\mu}}}\geq \kappa_1 ~  \mathrm{dist}(u,\mathcal{M})^2.
\end{equation*}
  On the other hand, Liu, Zhang and Zou \cite{LZZ23}, Piccione, Yang and Zhao \cite{PYZ25}, Yang and Zhao \cite{YZ252}, and Dai, Hu and Peng \cite{DHP25} have established the quantitative stability results for critical points
 of equation \eqref{limit2'}. For extensions to the stability of fractional HLS inequalities, we refer the readers to \cite{LYZ25}.

\subsection{Stability of Folland-Stein-Sobolev inequality on $\mathbb{H}^{n}$}
The Heisenberg group $\mathbb{H}^{n}$ is $\mathbb{C}^{n}\times\mathbb{R}$ with elements $\xi=(\xi_{l})=(z,t)$, $\xi'=(\xi'_{l})=(z',t')$, $1\leq l\leq 2n+1$, and group operation
\begin{equation*}
	\xi'\xi=\big(z+z',t+t'+2\mathrm{Im}z'\cdot\bar{z}\big).
\end{equation*}
The left translations are given by
\begin{equation*}
\tau_{\xi'}(\xi)=\xi'\xi,
\end{equation*}
and the dilations of group are
$\{\delta_{\lambda}\}_{\lambda>0}: \mathbb{H}^{n}\rightarrow\mathbb{H}^{n}$,
\begin{equation*}
\delta_{\lambda}(\xi)=(\lambda z,\lambda^{2}t).
\end{equation*}
Define the homogeneous norm
\begin{equation*}
|\xi|=(|z|^{4}+t^{2})^{\frac{1}{4}},
\end{equation*}
 and the distance
$d(\xi,\xi')=|(\xi')^{-1}\xi|$.
It holds that $|\delta_{\lambda}(\xi)|=\lambda|\xi|$ and
$d(\lambda\xi,\lambda\xi')=\lambda d(\xi,\xi')$.
As usual, the homogeneous dimension of $\mathbb{H}^{n}$ is $Q=2n+2$. Denote by $B_r(\xi_0)$ the ball of radius $r$ centered at $\xi_0$ with respect to the Heisenberg distance $d$, and write $B_r=B_r(\mathbf{0})$ when the center is the origin.

The canonical left-invariant vector fields on $\mathbb{H}^{n}$ are
\begin{equation*}
	X_{j}=\frac{\partial }{\partial x_{j}}+2y_{j}\frac{\partial }{\partial t},\quad X_{n+j}=\frac{\partial}{\partial y_{j}}-2x_{j}\frac{\partial}{\partial t},\quad j=1, \ldots, n.
\end{equation*}
It follows that the canonical left-invariant vector field is
\begin{equation*}
	\nabla_{\mathbb{H}}=(X_{1}, \ldots, X_{n},X_{n+1}, \ldots, X_{2n}),
\end{equation*}
and the Kohn Laplacian (or sub-Laplacian) operator is 
\begin{equation*}
	\Delta_{\mathbb{H}}=\mathop{\sum}\limits_{j=1}^{n}\big(X_{j}^{2}+X_{n+j}^{2}\big).
\end{equation*}
Let $Q^{\ast}=\frac{2Q}{Q-2}$, define the standard Folland-Stein-Sobolev space
\begin{equation*}
	S^{1,2}(\mathbb{H}^{n})=\Big\{u\in L^{Q^{\ast}}(\mathbb{H}^{n}):\nabla_{\mathbb{H}}u\in L^{2}(\mathbb{H}^{n})\Big\},
\end{equation*}
with the inner product
\begin{equation*}
\langle u,v\rangle_{S^{1,2}(\mathbb{H}^{n})}=\int_{\mathbb{H}^{n}}\nabla_{\mathbb{H}}u\cdot\nabla_{\mathbb{H}}v{d}\xi,
\end{equation*}
and the corresponding norm
\begin{equation*} \|u\|_{S^{1,2}(\mathbb{H}^{n})}=\left(\int_{\mathbb{H}^{n}}|\nabla_{\mathbb{H}}u|^{2}{d}\xi\right)^{\frac{1}{2}}.
\end{equation*}

In \cite{FS74}, Folland and Stein established the following well-known Sobolev type inequality on $\mathbb{H}^{n}$
\begin{equation}\label{folland-stein}
	\int_{\mathbb{H}^{n}}|\nabla_{\mathbb{H}}u|^{2}{d}\xi\geq S(Q)\left(\int_{\mathbb{H}^{n}}|u|^{Q^{\ast}}{d}\xi\right)^{\frac{2}{Q^{\ast}}},\quad \forall \, u\in S^{1,2}(\mathbb{H}^{n}).
\end{equation}
Moreover, it follows from Jerison and Lee \cite{JL88} 
that
\begin{equation*}
  S(Q)=\frac{4\pi n^2}{(n!4^n)^{\frac{1}{n+1}}},
\end{equation*}
and the corresponding extremal functions are in the form (Jerison-Lee bubbles):
\begin{equation}\label{gU}
  \mathfrak{g}_{\lambda,\xi_0}U(\xi)=\lambda^{\frac{Q-2}{2}}U(\delta_\lambda(\tau_{\xi_0}^{-1}(\xi))),
  ~~\lambda>0,~\zeta_{0}\in\mathbb{H}^{n},
\end{equation}
where
	\begin{equation*}
		U(\xi)=\frac{(2n)^n}{[(1+|z|^{2})^{2}+t^{2}]^{\frac{Q-2}{4}}}
	\end{equation*}
is the unique (up to translations and scalings, i.e., $\mathfrak{g}_{\lambda,\xi}U$ for all $\lambda>0$ and $\xi\in\mathbb{H}^{n}$)
positive solution of the Euler-Lagrange equation
\begin{equation}\label{limit1}
  -\Delta_{\mathbb{H}} u=|u|^{Q^*-2}u,~~\xi\in\mathbb{H}^{n}.
\end{equation}
 Malchiodi and Uguzzoni \cite{MU02} proved the non-degeneracy of the positive bubble solutions to  \eqref{limit1}.
 For further applications, Loiudice \cite{L05} derived the remainder term estimate for inequality
\eqref{folland-stein}, while Tang, Zhang and Zhang \cite{TZZ24} and Chen et al. \cite{CLTW25} established a positive upper and lower bound for the value of this remainder term estimate, respectively. For the fractional version, the readers may refer to
\cite{LZ15} for further details. More recently, Chen, Fan and Liao \cite{CFL25} obtained the quantitative stability result of critical points for equation \eqref{limit1}.

Similar to the HLS inequality in $\mathbb{R}^{N}$, the HLS inequality on the Heisenberg group
$\mathbb{H}^{n}$
has also been established by Folland and Stein \cite{FS74} as well as Frank and Lieb \cite{FL12}.
\begin{proposition}
	Suppose that  $Q\geq 4$, $0<\mu< Q$ and $t,r>1$ with $\frac{1}{t}+\frac{\mu}{Q}+\frac{1}{r}=2$, $f\in L^{t}(\mathbb{H}^{n})$ and $h\in L^{r}(\mathbb{H}^{n})$. Then there exists a constant $\widetilde{C}(Q,\mu,t,r)>0$ independent of $f$ and $h$ such that
	\begin{equation}\label{eq:HLSH} \int_{\mathbb{H}^{n}}\int_{\mathbb{H}^{n}}\frac{f(\xi)h(\eta)}{|\eta^{-1}\xi|^{\mu}}
d\xi{d}\eta\leq \widetilde{C}(Q,\mu,t,r)\|f\|_{L^t(\mathbb{H}^{n})}\|h\|_{L^r(\mathbb{H}^{n})}, ~~\xi,\eta\in\mathbb{H}^{n}.
	\end{equation}	
	If $t=r=\frac{2Q}{2Q-\mu}$, then
\begin{equation}\label{equH}
  \widetilde{C}(Q,\mu,t,r)=C(Q,\mu)=\bigg(\frac{\pi^{n+1}}{2^{n-1}n!}\bigg)^{\frac{\mu}{Q}}
  \frac{n!\Gamma\big(\frac{Q-\mu}{2}\big)}{\Gamma^2\big(\frac{2Q-\mu}{4}\big)},
\end{equation}
and there is equality in \eqref{eq:HLSH} if and only if $f\equiv (const.)h$ and
\begin{equation*}
  h(\xi)=c\lambda^{\frac{2Q-\mu}{2}}V(\delta_\lambda(\tau_{\xi_0}^{-1}(\xi))),
  ~~\lambda>0,~\zeta_{0}\in\mathbb{H}^{n}
\end{equation*}
for some $c\in \mathbb{C}$, 
where
\begin{equation*}
		V(\xi)=\frac{1}{[(1+|z|^{2})^{2}+t^{2}]^{\frac{2Q-\mu}{4}}}.
	\end{equation*}
\end{proposition}

Similarly, by \eqref{eq:HLSH}, $Q^{\ast}_{\mu}=\frac{2Q-\mu}{Q-2}$ calls the upper critical exponent.
For the upper critical case,  Yang and Zhang \cite{YZ251} and Zhang et al. \cite{ZWZLX25} proved: 
\begin{lemma}
	Let $Q\geq 4$, $0<\mu<Q$. Then for any $u\in S^{1,2}(\mathbb{H}^{n})\setminus\{0\}$, the inequality
	\begin{equation}\label{eq:HLS'}
		S_{HL}(Q,\mu)	\left(\int_{\mathbb{H}^{n}}\int_{\mathbb{H}^{n}}\frac{|u(\xi)|^{Q^{\ast}_{\mu}}|u(\eta)|
^{Q^{\ast}_{\mu}}}{|\eta^{-1}\xi|^{\mu}}{d}\xi{d}\eta\right)^{\frac{1}{Q^{\ast}_{\mu}}}\leq \|\nabla_{\mathbb{H}}u\|^2_{L^2(\mathbb{H}^{n})}
	\end{equation}
holds with the sharp constant
\begin{equation*}
  S_{HL}(Q,\mu)=S(Q)C(Q,\mu)^{-\frac{1}{Q_\mu^*}}.
\end{equation*}
The equality in \eqref{eq:HLS'} holds
if and only if
	\begin{equation*}
u(\xi)=c\mathfrak{g}_{\lambda,\xi_0}U(\xi)
	\end{equation*}
	for some $c\in\mathbb{C}\setminus\{0\}$, $\lambda>0$ and $\xi_0\in \mathbb{H}^n$, where $\mathfrak{g}_{\lambda,\xi_0}U$ is defined by \eqref{gU}. Moreover, 
$\mathfrak{g}_{\lambda,\xi}U$ (for all $\lambda>0$ and $\xi\in\mathbb{H}^{n}$) is the unique positive solution  of the Euler-Lagrange equation
\begin{equation}\label{limit3}
  -\Delta_{\mathbb{H}} u=\alpha(Q,\mu)\left(\int_{\mathbb{H}^{n}}\frac{|u(\eta)|
^{Q^{\ast}_{\mu}}}{|\eta^{-1}\xi|^{\mu}}{d}\eta\right)|u|^{Q_\mu^*-2}u,~~\xi,\eta\in\mathbb{H}^{n},
\end{equation}
where
\begin{equation*}
\alpha(Q,\mu)=S(Q)^{-\frac{Q-\mu}{2}}C(Q,\mu)^{-1}.
\end{equation*}
\end{lemma}
For equation \eqref{limit3}, Yang and Zhang \cite{YZ251} established
the following non-degeneracy property of bubbles. 
\begin{lemma}\label{nondege}
 Let $Q\geq 4$, $0<\mu< Q$.
 If $u \in S^{1,2}(\mathbb{H}^{n})$ is a solution to
  the linearized equation
  \begin{equation*}
	-\Delta_{\mathbb{H}} v=Q^{\ast}_{\mu}\left(\int_{\mathbb{H}^{n}}\frac{|U(\eta)|^{Q^{\ast}_{\mu}-1}v(\eta)}
{|\eta^{-1}\xi|^{\mu}}{d}\eta\right)|U|^{Q^{\ast}_{\mu}-2}U+(Q^{\ast}_{\mu}-1)
\left(\int_{\mathbb{H}^{n}}\frac{|U(\eta)|^{Q^{\ast}_{\mu}}}{|\eta^{-1}\xi|^{\mu}}{d}\eta\right)
|U|^{Q^{\ast}_{\mu}-2}v,
\end{equation*}
  then $u$ must be a linear combination of the functions  $\{Z^{a}\}_{a=1}^{2n +2}$, where $Z^a$ are defined by \eqref{Z1} and \eqref{Z2}.
\end{lemma}

In spite of the work Deng et al. \cite{DTYZ23} and Loiudice \cite{L05}, in this paper, we first are concerned with the remainder term of inequality \eqref{eq:HLS'} for $0<\mu< Q$ with $\mu\leq 4$.
Now, we present our first result as follows.
\begin{theorem}\label{main thm0}
Let $Q\geq 4$, $0<\mu< Q$ with $\mu\leq 4$.
 Then there exist two constants $A_2>A_1>0$ such that for every $u\in S^{1,2}(\mathbb{H}^{n})$, it holds that
 \begin{equation*}
A_2 ~  \mathrm{dist}(u,\mathfrak{M})^2\geq  \int_{\mathbb{H}^{n}}|\nabla_{\mathbb{H}} u|^{2}{d}\xi-S_{HL}(Q,\mu)	\left(\int_{\mathbb{H}^{n}}\int_{\mathbb{H}^{n}}\frac{|u(\xi)|^{Q^{\ast}_{\mu}}|u(\eta)|
^{Q^{\ast}_{\mu}}}{|\eta^{-1}\xi|^{\mu}}{d}\xi{d}\eta\right)^{\frac{1}{Q^{\ast}_{\mu}}}\geq A_1 ~  \mathrm{dist}(u,\mathfrak{M})^2,
\end{equation*}
where $\mathfrak{M}=\big\{c\mathfrak{g}_{\lambda,\xi}U: c\in \mathbb{C},\lambda>0,\xi\in \mathbb{H}^n\big\}$ is an $(2n+3)$-dimensional manifold and $\mathrm{dist}(u,\mathfrak{M})=\inf_{c\in \mathbb{C},\lambda>0,\xi\in \mathbb{H}^n}\|u-c\mathfrak{g}_{\lambda,\xi}U\|_{S^{1,2}(\mathbb{H}^{n})}$.
\end{theorem}


\begin{remark}
{\rm The restriction $\mu\leq 4$ is used to guarantee $Q_\mu^*\geq 2$, which is used in Lemma \ref{local}.}
\end{remark}

As a corollary of Theorem \ref{main thm0}, we will consider a remainder term inequality on bounded domains $\Omega\subset \mathbb{H}^n$. In sprite of the work of Brezis and Lieb \cite{BL85}, for each bounded domain $\Omega\subset \mathbb{H}^n$, we define the weak $L^q$-norm
\begin{equation}\label{wen}
  \|u\|_{L^q_w(\Omega)}=\sup\limits_{D\subset \Omega}\frac{\int_{D}|u|dx}{|D|^{\frac{q-1}{q}}}.
\end{equation}
In \cite{L06} (see also \cite{L05}), Loiudice  proved that there exists $A'>0$ such that
\begin{equation*}
  \|\nabla_\mathbb{H} u\|_{L^{2}(\Omega)}^{2}-S(Q)\|u\|_{L^{Q^{\ast}}(\Omega)}^{2}\geq A'  \|u\|_{L^{\frac{Q}{Q-2}}_w(\Omega)}^2,
  \quad \forall \, u\in{S}_0^{1,2}(\Omega),
\end{equation*}
where $S_0^{1,2}(\Omega)$ denotes the Folland-Stein-Sobolev
space defined as the completion of $C_c^\infty(\Omega)$ with respect to the norm $\|\cdot\|_{S^{1,2}(\mathbb{H}^n)}$.
Then, we establish the second-type remainder term for inequality \eqref{eq:HLS'} in a bounded domain, which is stated
as follows.
\begin{theorem}\label{main thm1}
Let $\Omega\subset \mathbb{H}^n$ be a bounded domain, $Q\geq 4$, $0<\mu< Q$ with $\mu\leq 4$.
 Then there exists $B'>0$ such that for every $u\in S_0^{1,2}(\Omega)$, it holds that
 \begin{equation*}
  \int_{\Omega}|\nabla_{\mathbb{H}} u|^{2}{d}\xi-S_{HL}(Q,\mu)	\left(\int_{\Omega}\int_{\Omega}\frac{|u(\xi)|^{Q^{\ast}_{\mu}}|u(\eta)|
^{Q^{\ast}_{\mu}}}{|\eta^{-1}\xi|^{\mu}}{d}\xi{d}\eta\right)^{\frac{1}{Q^{\ast}_{\mu}}}\geq B'  \|u\|_{L^{\frac{Q}{Q-2}}_w(\Omega)}^2,
\end{equation*}
where $L^{\frac{Q}{Q-2}}_w$ denotes the weak $L^{\frac{Q}{Q-2}}$-norm as in \eqref{wen}.
\end{theorem}

\begin{remark}
{\rm If $\mu=0$, then
Theorems \ref{main thm0} and \ref{main thm1} are exactly the conclusions obtained in \cite[Theorem 1.1]{L05} and \cite[Theorem 1.1]{L06}, respectively.
}
\end{remark}

Furthermore, we will also consider the quantitative
stability of critical points for equation \eqref{limit3}.
Denote $(S^{1,2}(\mathbb{H}^{n}))^{-1}$ as the dual space of $S^{1,2}(\mathbb{H}^{n})$, with the dual pairing
\begin{equation*}
  \langle u,v\rangle_{S^{1,2}(\mathbb{H}^{n}),(S^{1,2}(\mathbb{H}^{n}))^{-1}} =\int_{\mathbb{H}^{n}}uv d\xi,\quad \forall \, u\in S^{1,2}(\mathbb{H}^{n}),\ v\in (S^{1,2}(\mathbb{H}^{n}))^{-1}.
\end{equation*}
By using a refined asymptotic characterization in the sub-Riemannian setting of the Heisenberg group (see \cite[Theorem 1.1]{B08}), Zhang, Xu and Wang \cite{ZXW25} proved a nonlocal version of the global compactness result to \eqref{limit3} for nonnegative functions on the Heisenberg group:
\begin{theorem}\label{compactness thm}\cite[Theorem 1.8]{ZXW25}
Let $Q\geq4$, $\mu\in(0,Q)$ and $m\geq1$ be positive integers. Let
$\{u_{k}\}\subset S^{1,2}(\mathbb{H}^{n})$ be a sequence of nonnegative functions such that
\begin{equation*}
	\Big(m-\frac{1}{2}\Big)S_{HL}^{\frac{2Q-\mu}{Q+2-\mu}}\leq \|u_{k}\|_{S^{1,2}(\mathbb{H}^{n})}^{2}\leq \Big(m+\frac{1}{2}\Big)S_{HL}^{\frac{2Q-\mu}{Q+2-\mu}},
\end{equation*}
and
\begin{equation*}
\left\|\Delta_\mathbb{H} u_{k}+\left(\int_{\mathbb{H}^{n}}\frac{|u_k(\eta)|
^{Q^{\ast}_{\mu}}}{|\eta^{-1}\xi|^{\mu}}{d}\eta\right)|u_k|^{Q_\mu^*-2}u_k\right\|
_{(S^{1,2}(\mathbb{H}^{n}))^{-1}}\longrightarrow 0,\qquad\mathrm{as}~k\rightarrow+\infty.
\end{equation*}
Then there exists a sequence of parameters $\big\{\lambda_{i}^{(k)},\xi_{i}^{(k)}\big\}$ such that
\begin{equation*}
\left\|u_{k}-\mathop{\sum}\limits_{i=1}^{m}\mathfrak{g}_{\lambda_i^{(k)},\xi_i^{(k)}}U
\right\|_{S^{1,2}(\mathbb{H}^{n})}\longrightarrow0,\qquad\mathrm{as}~k\rightarrow+\infty.
\end{equation*}
Moreover, 
there exists $k_0\in \mathbb{N}$ such that bubbles $\big\{\mathfrak{g}_{\lambda_i^{(k)},\xi_i^{(k)}}\big\}_{k\geq k_0}$ is $\delta$-weakly interacting in the following sense: for $i\neq j$,
we define the quantity
\begin{equation*}
  \varepsilon_{ij}=\varepsilon(\lambda_i,\lambda_j,\xi_i,\xi_j)=\min\bigg\{\frac{\lambda_i}{\lambda_j},\frac{\lambda_j}{\lambda_i},
  \frac{1}{\lambda_i\lambda_j|\xi_i^{-1} \xi_j|^2}\bigg\},\quad \varepsilon=\max\limits_{i\neq j}\{\varepsilon_{ij}\},
\end{equation*}
then
\begin{equation*}
 \varepsilon^{(k)}= \max\limits_{i\neq j}\{\varepsilon_{ij}^{(k)}\}=\max\limits_{i\neq j}\Big\{\varepsilon\Big(\lambda_i^{(k)},\lambda_j^{(k)},\xi_i^{(k)},\xi_j^{(k)}\Big)\Big\}< \delta,\qquad \mathrm{for ~ all} ~ k\geq k_0.
\end{equation*}


\end{theorem}
Based on this result, using the non-degeneracy of bubbles, Zhang, Xu and Wang \cite{ZXW25} established the quantitative
stability of critical points for equation \eqref{limit3} in the single-bubble case when $Q\geq4$ and $\mu\in (0,Q)$, while the multi-bubble case when $Q=4$ and $\mu\in (0,2)$. More precisely, they proved that,
for a function $u\in S^{1,2}(\mathbb{H}^{n})$ that almost solves \eqref{limit3}, that is, $u$ closes to
the sum of $m\geq1$ weakly interacting  bubble solutions,
the $S^{1,2}(\mathbb{H}^{n})$-distance from $u$ to the manifold of sums of $m$ bubbles,
defined as
\begin{equation*}
\delta(u)=\left\|u-\sum\limits_{i=1}^m\mathfrak{g}_{\lambda_i ,\xi_i } U\right\|_{S^{1,2}(\mathbb{H}^{n})},
\end{equation*}
 can be linearly bounded by
\begin{equation*} \Upsilon(u)=\left\|\Delta_{\mathbb{H}}u+\left(\int_{\mathbb{H}^{n}}\frac{|u(\eta)|^{Q^{\ast}_{\mu}}}{|\eta^{-1}\xi|^{\mu}}{d}\eta\right)|u|^{Q^{\ast}_{\mu}-2}u\right\|_{(S^{1,2}(\mathbb{H}^{n}))^{-1}}.
	\end{equation*}
In the last part, using the method developed in \cite{CFL25}, we prove that for $Q=4$ and $\mu\in (2,4)$, the above quantitative stability estimate still holds.
\begin{theorem}\label{main thm}
Let $Q=4$, $\mu\in(2,4)$ and $m\geq2$ be positive integers.
There exists a constant $\tilde{\delta}=\tilde{\delta}(Q, m)>0$ such that for any $\delta \in (0, \tilde{\delta})$, the following holds:
   if $u\in S^{1,2}(\mathbb{H}^{n})$ satisfies
   \begin{equation*}
     \left\|u-\mathop{\sum}\limits_{i=1}^{m}\mathfrak{g}_{\tilde{\lambda}_i,\tilde{\xi}_i}U
\right\|_{S^{1,2}(\mathbb{H}^{n})}<\delta
   \end{equation*}
   for a family of $\delta$-weakly interacting bubbles $\big\{\mathfrak{g}_{\tilde{\lambda}_i , \tilde{\xi }_i} U\big\}_{i=1}^{m}$,
then there exist a family of bubbles $\big\{\mathfrak{g}_{\lambda_i ,\xi_i } U\big\}_{i=1}^m $ such that
\begin{equation}\label{main eq}
\delta(u) \leq C \Upsilon(u).
\end{equation}
Furthermore, $\varepsilon \leq C \Upsilon(u)$,
and estimate \eqref{main eq} is sharp in the sense that the power of $\Upsilon(u)$
 cannot be substituted with a larger (or smaller) one.
\end{theorem}

\begin{remark}
{\rm In Theorem \ref{main thm}, we actually assume that $Q-2<\mu<\min\{Q,4\}$,
there are several reasons for this:

(i) To guarantee $Q_\mu^*\geq2$, we need $\mu\leq4$.

(ii)  The condition $\mu >Q-2$ is necessary to obtain a good approximation for $\|\rho\|_{S^{1,2}(\mathbb{H}^{n})}$
  and   $\|h\|_{(S^{1,2}(\mathbb{H}^{n}))^{-1}}$.
  More precisely, in
 \eqref{main35} and \eqref{main42}, we need $\mu >Q-2$ to obtain  $O(\varepsilon^{\frac{\mu}{2}})=o\big(\varepsilon^{\frac{Q-2}{2}}\big)$.


 Combining (i) with (ii), we obtain $Q=4$ and $\mu\in(2,4)$.

Besides, $\mu<4$ is crucial to prove the invertibility for the operator $L$, we can see this in
\eqref{main52} and \eqref{main53}.
In addition,
to prove Theorem \ref{main thm}, in \eqref{maineq2}, we require that $\|N(\rho)\|_{(S^{1,2}(\mathbb{H}^{n}))^{-1}}$ can be controlled by $\vartheta\|\rho\|_{S^{1,2}(\mathbb{H}^{n})}$ for some $\vartheta\in (0,1)$,
this
 is true when $\mu<4$ by
  Lemma \ref{main1}. 
  }
\end{remark}

\begin{remark}
{\rm By considering a simple perturbation $u_\epsilon=(1+\epsilon)U$ for some small $\epsilon>0$, using \eqref{limit1} and \eqref{limit3}, we can find that the estimate \eqref{main eq} is sharp.}
\end{remark}



As a direct consequence of Theorems \ref{compactness thm}-\ref{main thm}, we have the following quantitative stability of global compactness for non-negative functions on $\mathbb{H}^n$.
\begin{corollary}\label{cor1-1}
Let $Q=4$, $\mu\in(2,4)$ and $m\geq2$ be positive integers. For any nonnegative function $u\in S^{1,2}(\mathbb{H}^{n})$ such that
\begin{equation*}
	\Big(m-\frac{1}{2}\Big)S_{HL}^{\frac{2Q-\mu}{Q+2-\mu}}\leq \|u\|_{S^{1,2}(\mathbb{H}^{n})}^{2}\leq \Big(m+\frac{1}{2}\Big)S_{HL}^{\frac{2Q-\mu}{Q+2-\mu}},
\end{equation*}
   there exist a family of bubbles $\big\{\mathfrak{g}_{\lambda_i , \xi_i } U\big\}_{i=1}^{m}$ such that \eqref{main eq} holds.
\end{corollary}

The paper is organized as follows. In Section
\ref{sec2}, we introduce some preliminary results. In Section \ref{sec3}, we prove Theorem \ref{main thm0}. Based on this result, we then complete the
proof of Theorem \ref{main thm1} in Section \ref{sec4}. Finally, in Section \ref{sec5}, we establish the quantitative stability estimate  of critical points for equation \eqref{limit3}, that is Theorem \ref{main thm}. Throughout the paper, the symbol $C$ denotes a positive constant possibly varies from line to line. The notation $X\lesssim Y$ ($X\gtrsim Y$) means that there exists a positive constant $C$ such that $X\leq CY$ ($X\geq CY$). We say $X \approx Y$ if both
$X\lesssim Y$ and $X\gtrsim Y$ hold.



\section{Preliminaries}\label{sec2}
For $\eta \in \mathbb{H}^n$ , let  $\eta^{(a)}$ denote its $a$-th coordinate.
we define
  \begin{equation}\label{Z1}
  Z^{a}=\frac{\partial \mathfrak{g}_{1, \eta}U }{\partial\eta^{(a)}}\bigg|_{ \eta=0},~~ a=1,\ldots, 2n+1.
  \end{equation}
  In addition, we define
\begin{equation}\label{Z2}
\begin{aligned}
Z^{2n+2}= \frac{\partial \mathfrak{g}_{r, 0}U }{\partial r}\bigg|_{r=1}=
\frac{Q-2}{2}U-(Q-2)\frac{|z|^2(1+|z|^2)+t^2}{(1+|z|^2)^2+t^2} U.
\end{aligned}
\end{equation}
  Then, the following estimate holds:
\begin{equation*}
|Z^{a}|\lesssim U,\quad \forall \, 1\leq a\leq  2n+2.
\end{equation*}

Let us consider the following eigenvalue problem for the linear perturbation operator:
\begin{align}\label{eig}
	&-\Delta_{\mathbb{H}} v+\left(\int_{\mathbb{H}^{n}}\frac{|U(\eta)|^{Q^{\ast}_{\mu}}}{|\eta^{-1}\xi|^{\mu}}{d}\eta\right)
|U|^{Q^{\ast}_{\mu}-2}v \nonumber\\
=&\Lambda\left[\left(\int_{\mathbb{H}^{n}}\frac{|U(\eta)|^{Q^{\ast}_{\mu}-1}v(\eta)}
{|\eta^{-1}\xi|^{\mu}}{d}\eta\right)|U|^{Q^{\ast}_{\mu}-2}U+
\left(\int_{\mathbb{H}^{n}}\frac{|U(\eta)|^{Q^{\ast}_{\mu}}}{|\eta^{-1}\xi|^{\mu}}{d}\eta\right)
|U|^{Q^{\ast}_{\mu}-2}v\right].
\end{align}
Then
following the work of Zhang, Xu and Wang \cite{ZXW25}, we can introduce the definition of eigenvalues of problem \eqref{eig} as the following.
\begin{definition}\label{defi}
{\rm The first eigenvalue of problem \eqref{eig} can be defined as
\begin{equation*}
  \Lambda_1=\inf \limits_{v\in S^{1,2}(\mathbb{H}^{n}) \backslash \{0\}}\frac{ \int_{\mathbb{H}^{n}}|\nabla_{\mathbb{H}} u|^{2}{d}\xi+ \int_{\mathbb{H}^{n}}\int_{\mathbb{H}^{n}}\frac{U^{Q^{\ast}_{\mu}}(\xi)U
^{Q^{\ast}_{\mu}-2}(\eta)v^2(\eta)}{|\eta^{-1}\xi|^{\mu}}{d}\xi{d}\eta}{ \int_{\mathbb{H}^{n}}\int_{\mathbb{H}^{n}}\frac{U^{Q^{\ast}_{\mu}-1}(\xi)v(\xi)U
^{Q^{\ast}_{\mu}-1}(\eta)v(\eta)}{|\eta^{-1}\xi|^{\mu}}{d}\xi{d}\eta+ \int_{\mathbb{H}^{n}}\int_{\mathbb{H}^{n}}\frac{U^{Q^{\ast}_{\mu}}(\xi)U
^{Q^{\ast}_{\mu}-2}(\eta)v^2(\eta)}{|\eta^{-1}\xi|^{\mu}}{d}\xi{d}\eta}.
\end{equation*}
Moreover, for any $l\in \mathbb{N}^+$, the $(l+1)$th-eigenvalues can be characterized as follows:
 \begin{equation*}
  \Lambda_{l+1}=\inf \limits_{v\in \mathbb{P}_{l+1} \backslash \{0\}}\frac{ \int_{\mathbb{H}^{n}}|\nabla_{\mathbb{H}} u|^{2}{d}\xi+ \int_{\mathbb{H}^{n}}\int_{\mathbb{H}^{n}}\frac{U^{Q^{\ast}_{\mu}}(\xi)U
^{Q^{\ast}_{\mu}-2}(\eta)v^2(\eta)}{|\eta^{-1}\xi|^{\mu}}{d}\xi{d}\eta}{ \int_{\mathbb{H}^{n}}\int_{\mathbb{H}^{n}}\frac{U^{Q^{\ast}_{\mu}-1}(\xi)v(\xi)U
^{Q^{\ast}_{\mu}-1}(\eta)v(\eta)}{|\eta^{-1}\xi|^{\mu}}{d}\xi{d}\eta+ \int_{\mathbb{H}^{n}}\int_{\mathbb{H}^{n}}\frac{U^{Q^{\ast}_{\mu}}(\xi)U
^{Q^{\ast}_{\mu}-2}(\eta)v^2(\eta)}{|\eta^{-1}\xi|^{\mu}}{d}\xi{d}\eta},
\end{equation*}
where
\begin{equation*}
  \mathbb{P}_{l+1}=\bigg\{v\in S^{1,2}(\mathbb{H}^{n}): \int_{\mathbb{H}^{n}}\nabla_{\mathbb{H}}v\cdot\nabla_{\mathbb{H}}e_j{d}\xi=0,\qquad \mathrm{for ~ all}~ j=1,\ldots,l\bigg\},
\end{equation*}
and $e_j$ is the corresponding eigenfunction of $\Lambda_j$.
}
\end{definition}
Then, by utilizing the non-degeneracy of bubbles, we have
\begin{lemma}\label{eigen}\cite[Lemma 2.3]{ZXW25}
Let $\Lambda_j$, $j=1,2,\ldots$, denote the eigenvalues of problem \eqref{eig} in the increasing order as in Definition \ref{defi}. Then
\begin{align*}
  \Lambda_1=&1, \quad \ \  e_1=\mathrm{Span}\{U\},\\
  \Lambda_2=&Q_\mu^*, \quad e_2=\mathrm{Span}\big\{Z^1,\ldots,Z^{2n+2}\big\}.
\end{align*}
Furthermore, $\Lambda_{j\geq3}>Q_\mu^*$.
\end{lemma}

We define the set of all gauge transformations associated with the standard bubble as
\begin{equation*}
\mathcal{G}=\big\{\mathfrak{g}_{\lambda, \xi}: \lambda >0, \xi \in \mathbb{H}^n\big\}.
\end{equation*}
It is straightforward to verify that each \(\mathfrak{g}_{\lambda, \xi} \in \mathcal{G}\) is an isometry on both $S^{1,2}(\mathbb{H}^{n})$ and $L^{Q^*}(\mathbb{H}^{n})$.
Moreover,
\(\mathcal{G}\) forms a group under composition.
For convenience, we will occasionally denote
  \begin{equation*}
 \mathfrak{g}_i= \mathfrak{g}_{\lambda_i, \xi_i}, \quad U_i=\mathfrak{g}_{\lambda_i, \xi_i} U, \quad Z_i^a= \mathfrak{g}_{\lambda_i, \xi_i} Z^a.
\end{equation*}
By direct calculations, we obtain 
\begin{equation}\label{rela}
\mathfrak{g}_i^{-1}\mathfrak{g}_{j}=\mathfrak{g}_{\frac{\lambda_j}{\lambda_i}, \delta_{\lambda_i}(\xi_i^{-1} \xi_j)}.
\end{equation}

\begin{lemma}\cite[Lemma 9.13]{TF07}\label{gg}
For a given sequence $\{\lambda_k,\xi_k \}_{k=1}^{\infty}$, the following statements are equivalent:
  \begin{itemize}
    \item For each $u\in S^{1,2}(\mathbb{H}^{n})$, $\mathfrak{g}_k u\rightharpoonup 0$ weakly in $S^{1,2}(\mathbb{H}^{n})$.
    \item For each $u\in S^{1,2}(\mathbb{H}^{n})$, $\mathfrak{g}_k^{-1} u\rightharpoonup 0$ weakly in $S^{1,2}(\mathbb{H}^{n})$.
    \item $|\log \lambda_k| +|\xi_k| \to +\infty.$
    \item $|\log \lambda_k| +\lambda_k |\xi_k| \to +\infty.$
  \end{itemize}
   In any of these cases, we say $\mathfrak{g}_k \rightharpoonup 0$.
\end{lemma}
\begin{remark}\label{iff}
As a consequence of \eqref{rela} and Lemma \ref{gg}, for sequences $\big\{\mathfrak{g}_{\lambda_1^{(k)}, \xi_1^{(k)}}\big\}_{k=1}^{\infty}$ and $\big\{\mathfrak{g}_{\lambda_2^{(k)}, \xi_2^{(k)}}\big\}_{k=1}^{\infty}$, we have
 \begin{equation*}
 \Big(\mathfrak{g}_{\lambda_1^{(k)}, \xi_1^{(k)}}\Big)^{-1}\mathfrak{g}_{\lambda_2^{(k)}, \xi_2^{(k)}} \rightharpoonup 0  ~ \mathrm{if ~ and ~ only ~ if} ~ \varepsilon_{12}^{(k)}\to 0,
\end{equation*}
where
\begin{equation*}
 \varepsilon_{12}^{(k)}= \min\left\{\frac{\lambda_1^{(k)}}{\lambda_2^{(k)}},\frac{\lambda_2^{(k)}}{\lambda_1^{(k)}},
  \frac{1}{\lambda_1^{(k)}\lambda_2^{(k)}|\big(\xi_1^{(k)}\big)^{-1} \xi_2^{(k)}|^2}\right\}.
\end{equation*}
\end{remark}

Moreover, we have the following elementary inequality.
\begin{lemma}\cite[Lemma 2.1]{LN98}\label{gs}
For any $a>0$, $b$ real, we have
\begin{align*}
 \big||a+b|^\beta-a^\beta\big|\leq C(\beta)
 \begin{cases}
             a^{\beta-1}|b|+|b|^\beta,\qquad & \mathrm{if} ~ \beta\geq1,\\
              \min\big\{a^{\beta-1}|b|,|b|^\beta\big\},\qquad & \mathrm{if} ~ 0<\beta<1,
           \end{cases}
  \end{align*}
  and
  \begin{align*}
 \big||a+b|^\beta(a+b)-a^{\beta+1}-(\beta+1)a^\beta b\big|\leq C(\beta)
 \begin{cases}
             a^{\beta-1}b^2+|b|^{\beta+1},\qquad & \mathrm{if} ~ \beta\geq1,\\
              \min\big\{a^{\beta-1}b^2,|b|^{\beta+1}\big\},\qquad & \mathrm{if} ~ 0\leq\beta<1.
           \end{cases}
  \end{align*}
\end{lemma}

\section{Proof of Theorem \ref{main thm0}}\label{sec3}
The main ingredient of the proof for Theorem \ref{main thm0} is contained in Lemma \ref{local}, where the behaviour of the sequences near $\mathfrak{M}$ is investigated.
Before this, we first establish a technical result.

\begin{lemma}\label{tech}
Let $Q\geq 4$, $0<\mu< Q$ with $\mu\leq 4$, then
\begin{equation*}
  S(Q)^{\frac{Q-\mu}{2}}C(Q,\mu)>1,
\end{equation*}
where $S(Q)$ is the sharp constant of the Folland-Stein-Sobolev inequality defined by \eqref{folland-stein} and $C(Q,\mu)$ is defined by \eqref{equH}.
\end{lemma}
\begin{proof}
Let $F(n,\mu)=S(Q)^{\frac{Q-\mu}{2}}C(Q,\mu)$, by a direct computation, we have
\begin{equation*}
  F(n,\mu)=\frac{\pi^{n+1}n^{2n+2-\mu}2^{2-\frac{\mu}{2}}\Gamma\big(n+1-\frac{\mu}{2}\big)}{\Gamma^2\big(n+1-\frac{\mu}{4}\big)}.
\end{equation*}
Then
\begin{equation*}
  \frac{F(n+1,\mu)}{F(n,\mu)}=\pi(n+1)^2\Big(\frac{n+1}{n}\Big)^{2n+2}
  \underbrace{\Big(\frac{n}{n+1}\Big)^
  \mu\frac{n+1-\frac{\mu}{2}}{\big(n+1-\frac{\mu}{4}\big)^2}}_{:=G(n,\mu)}.
\end{equation*}
It is easy to check that $G(n,\mu)$ is decreasing on $\mu\in (0,4]$ when $n\geq2$ and $\mu\in (0,4)$ when $n=1$. Hence, for $n\geq2$, we have
\begin{equation*}
  \frac{F(n+1,\mu)}{F(n,\mu)}\geq\pi(n+1)^2\Big(\frac{n+1}{n}\Big)^{2n+2}
  G(n,4)=
  \pi(n-1)\Big(\frac{n+1}{n}\Big)^{2n}>1.
\end{equation*}
Therefore, $F(n,\mu)$ is increasing on $n\geq2$, and
\begin{equation*}
  F(n,\mu)\geq \min\{F(1,\mu),F(2,\mu)\}=\min\left\{
  \underbrace{\frac{\pi^{2}2^{2-\frac{\mu}{2}}\Gamma\big(2-\frac{\mu}{2}\big)}{\Gamma^2\big(2-\frac{\mu}{4}\big)}\bigg|_{\mu<4}}_{:=A(\mu)},
  \underbrace{\frac{\pi^{3}2^{8-\frac{3\mu}{2}}\Gamma\big(3-\frac{\mu}{2}\big)}{\Gamma^2\big(3-\frac{\mu}{4}\big)}\bigg|_{\mu\leq 4}}_{:=B(\mu)}\right\}.
\end{equation*}

In the following, we prove that $A(\mu)>1$ and $B(\mu)>1$.
For $A(\mu)$, we have
\begin{equation*}
  \log A(\mu)=2\log \pi+2\log2-\frac{\mu}{2}\log 2+\log \Gamma\Big(2-\frac{\mu}{2}\Big)-2\log \Gamma\Big(2-\frac{\mu}{4}\Big),
\end{equation*}
then
\begin{equation*}
  \frac{A'(\mu)}{A(\mu)}=-\frac{1}{2}\log 2-\frac{1}{2}\frac{\Gamma'\big(2-\frac{\mu}{2}\big)}{\Gamma\big(2-\frac{\mu}{2}\big)}+
  \frac{1}{2}\frac{\Gamma'\big(2-\frac{\mu}{4}\big)}{\Gamma\big(2-\frac{\mu}{4}\big)}=-\frac{1}{2}\log 2-\frac{1}{2}\psi\Big(2-\frac{\mu}{2}\Big)+\frac{1}{2}\psi\Big(2-\frac{\mu}{4}\Big),
\end{equation*}
where $\psi(x)=\frac{\Gamma'(x)}{\Gamma(x)}$ is the Digamma function. Define
\begin{equation*}
  H(\mu)=\psi\Big(2-\frac{\mu}{4}\Big)-\psi\Big(2-\frac{\mu}{2}\Big)-\log 2,
\end{equation*}
then
\begin{equation*}
  H'(\mu)=-\frac{1}{4}\psi'\Big(2-\frac{\mu}{4}\Big)+\frac{1}{2}\psi'\Big(2-\frac{\mu}{2}\Big)=
  -\frac{1}{4}\psi_1\Big(2-\frac{\mu}{4}\Big)+\frac{1}{2}\psi_1\Big(2-\frac{\mu}{2}\Big),
\end{equation*}
where $\psi_1(x)=\psi'(x)$ is the Trigamma function. Since $\psi_1$ is positive and decreasing on $(0,+\infty)$, we have
\begin{equation*}
  H'(\mu)>
  -\frac{1}{4}\psi_1\Big(2-\frac{\mu}{4}\Big)+\frac{1}{4}\psi_1\Big(2-\frac{\mu}{2}\Big)=
  \frac{1}{4}\Big[\psi_1\Big(2-\frac{\mu}{2}\Big)-\psi_1\Big(2-\frac{\mu}{4}\Big)\Big]>0,
\end{equation*}
which implies that $H(\mu)$ is increasing on $(0,4)$.
Using the formula $\psi(x+1)=\psi(x)+\frac{1}{x}$ for any $x>0$, we obtain
\begin{equation*}
 \lim\limits _{\mu\rightarrow 0^+} H(\mu)=-\log 2<0,\quad
 \lim\limits _{\mu\rightarrow 4^-} H(\mu)=\lim\limits _{\varepsilon\rightarrow 0^+}[\psi(1+\varepsilon)-\psi(2\varepsilon)
-\log2] =+\infty>0.
\end{equation*}
Therefore, $A(\mu)$ is decreasing  on $(0,\mu_0)$, and increasing  on $(\mu_0,4)$ for some $\mu_0>0$ such that $H(\mu_0)=0$.
By using some numerical estimates, we obtain $\mu_0\approx 2.13$.
Thus
\begin{equation*}
  A(\mu)\geq A(\mu_0)\approx 25.07>1.
\end{equation*}
Similarly, 
we can prove that $B(\mu)$ is decreasing  on $(0,4]$. Thus
\begin{equation*}
  B(\mu)\geq B(4)=4\pi^3>1.
\end{equation*}
This ends the proof.
\end{proof}


Next, we investigate the behaviour of the sequences near $\mathfrak{M}$ for inequality \eqref{eq:HLS'}.
\begin{lemma}\label{local}
Let $Q\geq 4$, $0<\mu< Q$ with $\mu\leq 4$. For any sequence $\{u_k\}\subset S^{1,2}(\mathbb{H}^{n})\backslash \mathfrak{M}$ satisfying
\begin{equation*}
  \inf\limits_k \|u_k\|_{S^{1,2}(\mathbb{H}^{n})}>0,\quad \mathrm{dist}(u_k,\mathfrak{M})\rightarrow0,
\end{equation*}
we have
\begin{equation}\label{local1}
  \liminf\limits_{k\rightarrow+\infty}\frac{\int_{\mathbb{H}^{n}}|\nabla_{\mathbb{H}} u_k|^{2}{d}\xi-S_{HL}(Q,\mu)	\Big(\int_{\mathbb{H}^{n}}\int_{\mathbb{H}^{n}}\frac{|u_k(\xi)|^{Q^{\ast}_{\mu}}|u_k(\eta)|
^{Q^{\ast}_{\mu}}}{|\eta^{-1}\xi|^{\mu}}{d}\xi{d}\eta\Big)^{\frac{1}{Q^{\ast}_{\mu}}}}{\mathrm{dist}(u_k,\mathfrak{M})^2}\geq A_1,
\end{equation}
and
\begin{equation}\label{local2}
  \limsup\limits_{k\rightarrow+\infty}\frac{\int_{\mathbb{H}^{n}}|\nabla_{\mathbb{H}} u_k|^{2}{d}\xi-S_{HL}(Q,\mu)	\Big(\int_{\mathbb{H}^{n}}\int_{\mathbb{H}^{n}}\frac{|u_k(\xi)|^{Q^{\ast}_{\mu}}|u_k(\eta)|
^{Q^{\ast}_{\mu}}}{|\eta^{-1}\xi|^{\mu}}{d}\xi{d}\eta\Big)^{\frac{1}{Q^{\ast}_{\mu}}}}{\mathrm{dist}(u_k,\mathfrak{M})^2}\leq 1.
\end{equation}
\end{lemma}
\begin{proof}
Let $d_k={\rm dist}(u_k,\mathfrak{M})=\inf_{c\in \mathbb{C},\lambda>0,\xi\in \mathbb{H}^n}\|u_k-c\mathfrak{g}_{\lambda,\xi}U\|_{S^{1,2}(\mathbb{H}^{n})}$, then $d_k\rightarrow0$ as $k\to +\infty$. It is well known that for each $u_k\in {S}^{1,2}(\mathbb{H}^n)$, there exists $(c_k,\lambda_k,\xi_k)\in \mathbb{C}\backslash \{0\}\times\mathbb{R}^+\times\mathbb{H}^n$ such that
\begin{equation*}
  d_k=\big\|u_k-c_k\mathfrak{g}_{\lambda_k,\xi_k}U\big\|_{{S}^{1,2}(\mathbb{H}^n)}.
\end{equation*}
    Since $\mathfrak{M}\backslash\{0\}$ is an $(2n+3)$-dimensional smooth  manifold
    embedded in ${S}^{1,2}(\mathbb{H}^n)$,
    we have
    \begin{equation}\label{perpen}
      \big(u_k-c_k\mathfrak{g}_{\lambda_k,\xi_k}U\big)\perp T_{c_k\mathfrak{g}_{\lambda_k,\xi_k}U}\mathfrak{M},
    \end{equation}
    where the tangent space at $(c_k,\lambda_k,\xi_k)$ is given by (see Lemma \ref{eigen})
    \begin{equation*}
      T_{c_k\mathfrak{g}_{\lambda_k,\xi_k}U}\mathfrak{M}={\rm Span}\left\{\mathfrak{g}_{\lambda_k,\xi_k}U, \frac{\partial \mathfrak{g}_{r,\xi_k}U}{\partial r}\bigg|_{r=\lambda_k},\frac{\partial \mathfrak{g}_{\lambda_k,\eta}U}{\partial \eta^{(a)}}\bigg|_{\eta=\xi_k},~a=1,\ldots,2n+1
      \right\}.
    \end{equation*}
    Let
    \begin{equation*}
      u_k=c_k\mathfrak{g}_{\lambda_k,\xi_k}U+d_k w_k,
    \end{equation*}
     then $w_k$ is perpendicular to  $T_{c_k\mathfrak{g}_{\lambda_k,\xi_k}U}\mathfrak{M}$, $\|w_k\|_{{S}^{1,2}(\mathbb{H}^n)}=1$ and
    \begin{equation*}
    \|u_k\|^2_{{S}^{1,2}(\mathbb{H}^n)}=d_k^2\|w_k\|_{{S}^{1,2}(\mathbb{H}^n)}^2+c_k^2\|\mathfrak{g}_{\lambda_k,\xi}U\|^2_{{S}^{1,2}(\mathbb{H}^n)}
    =d_k^2+c_k^2\|U\|^2_{{S}^{1,2}(\mathbb{H}^n)},
    \end{equation*}
    where we have used the fact that
    \begin{equation*}
      \|\mathfrak{g}_{\lambda_k,\xi}U\|_{{S}^{1,2}(\mathbb{H}^n)}=\|U\|_{{S}^{1,2}(\mathbb{H}^n)}.
    \end{equation*}
    Since $Q_\mu^*\geq2$, by the orthogonality from above, we get
    \begin{align}\label{local3}
    \int_{\mathbb{H}^{n}}\int_{\mathbb{H}^{n}}\frac{|u_k|^{Q^{\ast}_{\mu}}|u_k|
^{Q^{\ast}_{\mu}}}{|\eta^{-1}\xi|^{\mu}}&{d}\xi{d}\eta
  =  c_k^{2\cdot Q_{\mu}^{\ast}} \int_{\mathbb{H}^{n}}\int_{\mathbb{H}^{n}}\frac{|\mathfrak{g}_{\lambda_k,\xi_k}U|^{Q^{\ast}_{\mu}}|\mathfrak{g}_{\lambda_k,\xi_k}U|
^{Q^{\ast}_{\mu}}}{|\eta^{-1}\xi|^{\mu}}{d}\xi{d}\eta
\nonumber  \\&  + Q_{\mu}^{\ast}(Q_{\mu}^{\ast}-1)c_k^{2(Q_{\mu}^{\ast}-1)}d_k^2 \underbrace{\int_{\mathbb{H}^{n}}\int_{\mathbb{H}^{n}}\frac{|\mathfrak{g}_{\lambda_k,\xi_k}U|^{Q^{\ast}_{\mu}}|\mathfrak{g}_{\lambda_k,\xi_k}U|
^{Q^{\ast}_{\mu}-2} w_k^2}{|\eta^{-1}\xi|^{\mu}}{d}\xi{d}\eta}_{:=P_{k,1}} \nonumber\\
    & + (Q_{\mu}^{\ast})^2c_k^{2(Q_{\mu}^{\ast}-1)}d_k^2 \underbrace{\int_{\mathbb{H}^{n}}\int_{\mathbb{H}^{n}}\frac{|\mathfrak{g}_{\lambda_k,\xi_k}U|^{Q^{\ast}_{\mu}-1}w_k|\mathfrak{g}_{\lambda_k,\xi_k}U|
^{Q^{\ast}_{\mu}-1}w_k}{|\eta^{-1}\xi|^{\mu}}{d}\xi{d}\eta}_{:=P_{k,2}}+ o(d_k^2),
    \end{align}
    due to
    \begin{equation*}
      \int_{\mathbb{H}^{n}}\int_{\mathbb{H}^{n}}\frac{|\mathfrak{g}_{\lambda_k,\xi_k}U|^{Q^{\ast}_{\mu}}|\mathfrak{g}_{\lambda_k,\xi_k}U|
^{Q^{\ast}_{\mu}-1}w_k}{|\eta^{-1}\xi|^{\mu}}{d}\xi{d}\eta=\alpha(Q,\mu)^{-1}\int_{\mathbb{H}^{n}} \nabla_\mathbb{H}\mathfrak{g}_{\lambda_k,\xi_k}U\cdot \nabla_\mathbb{H}w_k d\xi=0.
    \end{equation*}
By \eqref{folland-stein} and \eqref{eq:HLSH}, using the 
 H\"{o}lder inequality, 
we obtain
\begin{align}\label{local4}
  P_{k,1}\leq & C(Q,\mu)\|\mathfrak{g}_{\lambda_k,\xi_k}U\|_{L^{Q^*}(\mathbb{H}^{n})}^{2(Q_\mu^*-1)}\|w_k\|_{L^{Q^*}(\mathbb{H}^{n})}^{2}
 \nonumber \\=&C(Q,\mu)\|U\|_{L^{Q^*}(\mathbb{H}^{n})}^{2(Q_\mu^*-1)}\|w_k\|_{L^{Q^*}(\mathbb{H}^{n})}^{2}
=C(Q,\mu)S(Q)^{\frac{Q+2-\mu}{2}} \|w_k\|_{L^{Q^*}(\mathbb{H}^{n})}^{2} \nonumber\\
\leq& S(Q)^{\frac{Q-\mu}{2}}C(Q,\mu)\|w_k\|^2_{S^{1,2}(\mathbb{H}^{n})}
=S(Q)^{\frac{Q-\mu}{2}}C(Q,\mu),
\end{align}
and
\begin{equation*}
  P_{k,2}\leq  S(Q)^{\frac{Q-\mu}{2}}C(Q,\mu),
\end{equation*}
since
    \begin{equation*}
      \|\mathfrak{g}_{\lambda_k,\xi_k}U\|_{L^{Q^*}(\mathbb{H}^{n})}=\|U\|_{L^{Q^*}(\mathbb{H}^{n})}.
    \end{equation*}
    Now, we prove \eqref{local1}, and the proof is divided into four cases.

\textbf{Case 1:} $P_{k,1}=P_{k,2}=o_k(1)$. In this case, from \eqref{local3},
similar to \eqref{local4},
we have
\begin{align*}
    \int_{\mathbb{H}^{n}}\int_{\mathbb{H}^{n}}\frac{|u_k|^{Q^{\ast}_{\mu}}|u_k|
^{Q^{\ast}_{\mu}}}{|\eta^{-1}\xi|^{\mu}}{d}\xi{d}\eta
  \leq &c_k^{2\cdot Q_{\mu}^{\ast}} \int_{\mathbb{H}^{n}}\int_{\mathbb{H}^{n}}\frac{|\mathfrak{g}_{\lambda_k,\xi_k}U|^{Q^{\ast}_{\mu}}|\mathfrak{g}_{\lambda_k,\xi_k}U|
^{Q^{\ast}_{\mu}}}{|\eta^{-1}\xi|^{\mu}}{d}\xi{d}\eta+ o(d_k^2) \nonumber\\
\leq& c_k^{2\cdot Q_{\mu}^{\ast}} S(Q)^{\frac{Q-\mu}{2}}C(Q,\mu)\|U\|^2_{S^{1,2}(\mathbb{H}^{n})}+o(d_k^2).
    \end{align*}
    Thus it follows that
    \begin{align*}
      \left(\int_{\mathbb{H}^{n}}\int_{\mathbb{H}^{n}}\frac{|u_k|^{Q^{\ast}_{\mu}}|u_k|
^{Q^{\ast}_{\mu}}}{|\eta^{-1}\xi|^{\mu}}{d}\xi{d}\eta\right)^{\frac{1}{Q_\mu^*}}\leq& \bigg(c_k^{2\cdot Q_{\mu}^{\ast}}  S(Q)^{\frac{Q-\mu}{2}}C(Q,\mu) \|U\|^2_{S^{1,2}(\mathbb{H}^{n})}+o(d_k^2)\bigg)^{\frac{1}{Q_\mu^*}}\\
\leq& c_k^{2} \Big( S(Q)^{\frac{Q-\mu}{2}}C(Q,\mu) \Big)^{\frac{1}{Q_\mu^*}}\|U\|^{\frac{2}{Q_\mu^*}}_{S^{1,2}(\mathbb{H}^{n})}+o(d_k^2),
    \end{align*}
    then
    \begin{align*}
     & \int_{\mathbb{H}^{n}}|\nabla_{\mathbb{H}} u_k|^{2}{d}\xi-S_{HL}(Q,\mu)	\left(\int_{\mathbb{H}^{n}}\int_{\mathbb{H}^{n}}\frac{|u_k|^{Q^{\ast}_{\mu}}|u_k|
^{Q^{\ast}_{\mu}}}{|\eta^{-1}\xi|^{\mu}}{d}\xi{d}\eta\right)^{\frac{1}{Q^{\ast}_{\mu}}}\\
\geq&
d_k^2+c_k^2\left[\|U\|^2_{{S}^{1,2}(\mathbb{H}^n)}-S_{HL}(Q,\mu)\Big( S(Q)^{\frac{Q-\mu}{2}}C(Q,\mu) \Big)^{\frac{1}{Q_\mu^*}} \|U\|^{\frac{2}{Q_\mu^*}}_{S^{1,2}(\mathbb{H}^{n})}\right]+o(d_k^2)
=d_k^2+o(d_k^2),
    \end{align*}
    thanks to
    \begin{align*}
      &S_{HL}(Q,\mu)\big( S(Q)^{\frac{Q-\mu}{2}}C(Q,\mu) \big)^{\frac{1}{Q_\mu^*}} \|U\|^{\frac{2}{Q_\mu^*}-2}_{S^{1,2}(\mathbb{H}^{n})}\\
      =&
      S(Q)C(Q,\mu)^{-\frac{1}{Q_\mu^*}}
      \Big(S(Q)^{\frac{Q-\mu}{2}}C(Q,\mu)\Big)^{\frac{Q-2}{2Q-\mu}}
      S(Q)^{-\frac{Q(Q+2-\mu)}{2(2Q-\mu)}}=1.
    \end{align*}
    Choosing $d_k$ small enough, we have $C\in (0,1)$ such that
    \begin{equation*}
      \int_{\mathbb{H}^{n}}|\nabla_{\mathbb{H}} u_k|^{2}{d}\xi-S_{HL}(Q,\mu)	\left(\int_{\mathbb{H}^{n}}\int_{\mathbb{H}^{n}}\frac{|u_k|^{Q^{\ast}_{\mu}}|u_k|
^{Q^{\ast}_{\mu}}}{|\eta^{-1}\xi|^{\mu}}{d}\xi{d}\eta\right)^{\frac{1}{Q^{\ast}_{\mu}}}\geq Cd_k^2,
    \end{equation*}
    which proves \eqref{local1}.

    \textbf{Case 2:} $P_{k,1},P_{k,2}\geq \widetilde{C}$ for some $\widetilde{C}>0$. In this case, by \eqref{perpen},
 the definition of $\Lambda_3$ implies that
    \begin{equation*}
    \int_{\mathbb{H}^{n}}|\nabla_{\mathbb{H}} w_k|^{2}{d}\xi+P_{k,1}\geq \Lambda_3(P_{k,1}+P_{k,2}),
    \end{equation*}
    thus
    \begin{equation*}
    1\geq (\Lambda_3-1)P_{k,1}+\Lambda_3 P_{k,2}.
    \end{equation*}
    Then from \eqref{local3} and $\Lambda_3>Q_\mu^*$, we can derive that
    \begin{align*}
    \int_{\mathbb{H}^{n}}\int_{\mathbb{H}^{n}}\frac{|u_k|^{Q^{\ast}_{\mu}}|u_k|
^{Q^{\ast}_{\mu}}}{|\eta^{-1}\xi|^{\mu}}{d}\xi{d}\eta
  \leq &c_k^{2\cdot Q_{\mu}^{\ast}} \int_{\mathbb{H}^{n}}\int_{\mathbb{H}^{n}}\frac{|\mathfrak{g}_{\lambda_k,\xi_k}U|^{Q^{\ast}_{\mu}}|\mathfrak{g}_{\lambda_k,\xi_k}U|
^{Q^{\ast}_{\mu}}}{|\eta^{-1}\xi|^{\mu}}{d}\xi{d}\eta \nonumber\\
&+Q_{\mu}^{\ast}(Q_{\mu}^{\ast}-\Lambda_3)c_k^{2(Q_{\mu}^{\ast}-1)} (P_{k,1}+P_{k,2})d_k^2\\
&+Q_{\mu}^{\ast}c_k^{2(Q_{\mu}^{\ast}-1)} \big[(\Lambda_3-1)P_{k,1}+\Lambda_3P_{k,2}\big]d_k^2+ o(d_k^2)\\
\leq& c_k^{2\cdot Q_{\mu}^{\ast}} S(Q)^{\frac{Q-\mu}{2}}C(Q,\mu)\|U\|^2_{S^{1,2}(\mathbb{H}^{n})}\\
&+Q_{\mu}^{\ast}c_k^{2(Q_{\mu}^{\ast}-1)}
\big[2\widetilde{C}(Q_{\mu}^{\ast}-\Lambda_3)+1\big]d_k^2+o(d_k^2).
    \end{align*}
    Thus, we have
    \begin{align*}
      &\left(\int_{\mathbb{H}^{n}}\int_{\mathbb{H}^{n}}\frac{|u_k|^{Q^{\ast}_{\mu}}|u_k|
^{Q^{\ast}_{\mu}}}{|\eta^{-1}\xi|^{\mu}}{d}\xi{d}\eta\right)^{\frac{1}{Q_\mu^*}}\\
\leq& \bigg(c_k^{2\cdot Q_{\mu}^{\ast}}  S(Q)^{\frac{Q-\mu}{2}}C(Q,\mu) \|U\|^2_{S^{1,2}(\mathbb{H}^{n})}+Q_{\mu}^{\ast}c_k^{2(Q_{\mu}^{\ast}-1)}
\big[2\widetilde{C}(Q_{\mu}^{\ast}-\Lambda_3)+1\big]d_k^2\bigg)^{\frac{1}{Q_\mu^*}}+o(d_k^2)\\
\leq& c_k^{2} \Big( S(Q)^{\frac{Q-\mu}{2}}C(Q,\mu) \Big)^{\frac{1}{Q_\mu^*}}\|U\|^{\frac{2}{Q_\mu^*}}_{S^{1,2}(\mathbb{H}^{n})}\\
&+\Big( S(Q)^{\frac{Q-\mu}{2}}C(Q,\mu) \Big)^{\frac{1}{Q_\mu^*}-1}\|U\|^{\frac{2}{Q_\mu^*}-2}_{S^{1,2}(\mathbb{H}^{n})}\big[2\widetilde{C}(Q_{\mu}^{\ast}-\Lambda_3)+1\big]d_k^2+o(d_k^2),
    \end{align*}
    which leads to
    \begin{align*}
     & \int_{\mathbb{H}^{n}}|\nabla_{\mathbb{H}} u_k|^{2}{d}\xi-S_{HL}(Q,\mu)	\left(\int_{\mathbb{H}^{n}}\int_{\mathbb{H}^{n}}\frac{|u_k|^{Q^{\ast}_{\mu}}|u_k|
^{Q^{\ast}_{\mu}}}{|\eta^{-1}\xi|^{\mu}}{d}\xi{d}\eta\right)^{\frac{1}{Q^{\ast}_{\mu}}}\\
\geq&
d_k^2\left[1-S_{HL}(Q,\mu)\Big( S(Q)^{\frac{Q-\mu}{2}}C(Q,\mu) \Big)^{\frac{1}{Q_\mu^*}-1}\big[2\widetilde{C}(Q_{\mu}^{\ast}-\Lambda_3)+1\big]\|U\|^{\frac{2}{Q_\mu^*}-2}_{S^{1,2}(\mathbb{H}^{n})}
\right]\\
&+c_k^2\left[\|U\|^2_{{S}^{1,2}(\mathbb{H}^n)}-S_{HL}(Q,\mu)\Big( S(Q)^{\frac{Q-\mu}{2}}C(Q,\mu) \Big)^{\frac{1}{Q_\mu^*}} \|U\|^{\frac{2}{Q_\mu^*}}_{S^{1,2}(\mathbb{H}^{n})}\right]+o(d_k^2)\\
>&2\widetilde{C}(\Lambda_3-Q_{\mu}^{\ast})d_k^2+o(d_k^2),
    \end{align*}
    since $S(Q)^{\frac{Q-\mu}{2}}C(Q,\mu) >1$ by Lemma \ref{tech}. Choosing $\widetilde{C}>0$ small enough such that $2\widetilde{C}(\Lambda_3-Q_{\mu}^{\ast})<1$, we conclude \eqref{local1}.

    \textbf{Case 3:} $P_{k,1}=o_k(1)$, $P_{k,2}\geq \widetilde{C}>0$. With a similar argument of Case 2, we obtain \eqref{local1}.

    \textbf{Case 4:} $P_{k,1}\geq \widetilde{C}>0$, $P_{k,2}=o_k(1)$. With a similar argument of Case 2, we obtain \eqref{local1}.

    Next we show that \eqref{local2}  holds true. Thanks to \eqref{local3}, we know
    \begin{align*}
      \left(\int_{\mathbb{H}^{n}}\int_{\mathbb{H}^{n}}\frac{|u_k|^{Q^{\ast}_{\mu}}|u_k|
^{Q^{\ast}_{\mu}}}{|\eta^{-1}\xi|^{\mu}}{d}\xi{d}\eta\right)^{\frac{1}{Q_\mu^*}}
\geq c_k^{2} \Big( S(Q)^{\frac{Q-\mu}{2}}C(Q,\mu) \Big)^{\frac{1}{Q_\mu^*}}\|U\|^{\frac{2}{Q_\mu^*}}_{S^{1,2}(\mathbb{H}^{n})}+o(d_k^2).
    \end{align*}
    Therefore,
    \begin{align*}
     & \int_{\mathbb{H}^{n}}|\nabla_{\mathbb{H}} u_k|^{2}{d}\xi-S_{HL}(Q,\mu)	\left(\int_{\mathbb{H}^{n}}\int_{\mathbb{H}^{n}}\frac{|u_k|^{Q^{\ast}_{\mu}}|u_k|
^{Q^{\ast}_{\mu}}}{|\eta^{-1}\xi|^{\mu}}{d}\xi{d}\eta\right)^{\frac{1}{Q^{\ast}_{\mu}}}\\
\leq&
d_k^2+c_k^2\bigg[\|U\|^2_{{S}^{1,2}(\mathbb{H}^n)}-S_{HL}(Q,\mu)\Big( S(Q)^{\frac{Q-\mu}{2}}C(Q,\mu) \Big)^{\frac{1}{Q_\mu^*}} \|U\|^{\frac{2}{Q_\mu^*}}_{S^{1,2}(\mathbb{H}^{n})}\bigg]+o(d_k^2)
=d_k^2+o(d_k^2),
    \end{align*}
    then \eqref{local2} follows immediately.
\end{proof}

Now, we are ready to prove Theorem \ref{main thm0}.

\begin{proof}
[Proof of Theorem \ref{main thm0}]
We argue by contradiction. In fact, if the theorem is false, then there exists a sequence $\{u_k\}\subset S^{1,2}(\mathbb{H}^{n})\backslash \mathfrak{M}$ such that
\begin{equation}\label{local10}
  \frac{\int_{\mathbb{H}^{n}}|\nabla_{\mathbb{H}} u_k|^{2}{d}\xi-S_{HL}(Q,\mu)	\Big(\int_{\mathbb{H}^{n}}\int_{\mathbb{H}^{n}}\frac{|u_k(\xi)|^{Q^{\ast}_{\mu}}|u_k(\eta)|
^{Q^{\ast}_{\mu}}}{|\eta^{-1}\xi|^{\mu}}{d}\xi{d}\eta\Big)^{\frac{1}{Q^{\ast}_{\mu}}}}{\mathrm{dist}
(u_k,\mathfrak{M})^2}\longrightarrow0,
\end{equation}
or
\begin{equation}\label{local20}
  \frac{\int_{\mathbb{H}^{n}}|\nabla_{\mathbb{H}} u_k|^{2}{d}\xi-S_{HL}(Q,\mu)	\Big(\int_{\mathbb{H}^{n}}\int_{\mathbb{H}^{n}}\frac{|u_k(\xi)|^{Q^{\ast}_{\mu}}|u_k(\eta)|
^{Q^{\ast}_{\mu}}}{|\eta^{-1}\xi|^{\mu}}{d}\xi{d}\eta\Big)^{\frac{1}{Q^{\ast}_{\mu}}}}
{\mathrm{dist}(u_k,\mathfrak{M})^2}\longrightarrow+\infty,
\end{equation}
as $k\rightarrow+\infty$. By homogeneity, we can assume that $\|u_k\|_{S^{1,2}(\mathbb{H}^{n})}=1$, and up to a subsequence, we assume that $\mathrm{dist}(u_k,\mathfrak{M})\rightarrow \tau\in [0,1]$ since
\begin{equation*}
\mathrm{dist}(u_k,\mathfrak{M})=\inf_{c\in \mathbb{C},\lambda>0,\xi\in \mathbb{H}^n}\|u_k-c\mathfrak{g}_{\lambda,\xi}U\|_{S^{1,2}(\mathbb{H}^{n})}\leq \|u_k\|_{S^{1,2}(\mathbb{H}^{n})}=1.
\end{equation*}
 If \eqref{local20} holds, it follows that $\tau=0$, which contradicts Lemma \ref{local}.
 Moreover, if \eqref{local10} holds, it
also leads to a contradiction with Lemma \ref{local} when
 $\tau=0$. Consequently, the only remaining possibility is
that \eqref{local10} holds and $\xi\in (0,1]$, that is,
\begin{equation*}
  \mathrm{dist}(u_k,\mathfrak{M}) \rightarrow \tau>0,\quad \int_{\mathbb{H}^{n}}|\nabla_{\mathbb{H}} u_k|^{2}{d}\xi-S_{HL}(Q,\mu)	\left(\int_{\mathbb{H}^{n}}\int_{\mathbb{H}^{n}}\frac{|u_k(\xi)|^{Q^{\ast}_{\mu}}|u_k(\eta)|
^{Q^{\ast}_{\mu}}}{|\eta^{-1}\xi|^{\mu}}{d}\xi{d}\eta\right)^{\frac{1}{Q^{\ast}_{\mu}}}\longrightarrow0,
\end{equation*}
as $k\rightarrow+\infty$. Then, we have
\begin{equation*}
 \|u_k\|_{S^{1,2}(\mathbb{H}^{n})}=1,\quad \left(\int_{\mathbb{H}^{n}}\int_{\mathbb{H}^{n}}\frac{|u_k(\xi)|^{Q^{\ast}_{\mu}}|u_k(\eta)|
^{Q^{\ast}_{\mu}}}{|\eta^{-1}\xi|^{\mu}}{d}\xi{d}\eta\right)^{\frac{1}{Q^{\ast}_{\mu}}}\longrightarrow \frac{1}{S_{HL}(Q,\mu)},\qquad \mathrm{as} ~k\rightarrow+\infty.
\end{equation*}
By a concentration-compactness principle (see \cite[Theorem 1.3]{DGY22}) suitably adapted to the Heisenberg setting, there exist $\lambda_k>0$, $\xi_k\in \mathbb{H}^n$ and some $U_0\in \mathfrak{M}$ such that
\begin{equation*}
  \mathfrak{g}_{\lambda_k,\xi_k}u_k\rightarrow U_0\qquad \mathrm{in}~ S^{1,2}(\mathbb{H}^{n})~\mathrm{as} ~k\rightarrow+\infty.
\end{equation*}
Observe that $\mathrm{dist}(u_k,\mathfrak{M})=\mathrm{dist}(\mathfrak{g}_{\lambda_k,\xi_k}u_k,\mathfrak{M})$, we have $\mathrm{dist}(u_k,\mathfrak{M})\rightarrow0$ as $k\rightarrow+\infty$,
a contradiction.
\end{proof}

\section{Proof of Theorem \ref{main thm1}}\label{sec4}
In this section, we show how Theorem \ref{main thm1} can be derived as a consequence of
Theorem \ref{main thm0}.

\begin{proof}
[Proof of Theorem \ref{main thm1}]
Assume by contradiction that the theorem is not true, then there exists a sequence  $\{u_k\}\subset S_0^{1,2}(\Omega)$ with $\|u_k\|_{S^{1,2}(\Omega)}=1$ such that
\begin{equation*}
  \frac{\int_{\Omega}|\nabla_{\mathbb{H}} u_k|^{2}{d}\xi-S_{HL}(Q,\mu)	\Big(\int_{\Omega}\int_{\Omega}\frac{|u_k(\xi)|^{Q^{\ast}_{\mu}}|u_k(\eta)|
^{Q^{\ast}_{\mu}}}{|\eta^{-1}\xi|^{\mu}}{d}\xi{d}\eta\Big)^{\frac{1}{Q^{\ast}_{\mu}}}}
{\|u_k\|_{L^{\frac{Q}{Q-2}}_w(\Omega)}^2}\longrightarrow0,\qquad \mathrm{as} ~k\rightarrow+\infty.
\end{equation*}
Since $\|u_k\|_{L^{\frac{Q}{Q-2}}_w(\Omega)}\leq \|u_k\|_{L^{\frac{Q}{Q-2}}(\Omega)}\leq C_1\|u_k\|_{L^{\frac{2Q}{Q-2}}(\Omega)}\leq C_2$, we have
\begin{equation*}
  \int_{\Omega}|\nabla_{\mathbb{H}} u_k|^{2}{d}\xi-S_{HL}(Q,\mu)	\left(\int_{\Omega}\int_{\Omega}\frac{|u_k(\xi)|^{Q^{\ast}_{\mu}}|u_k(\eta)|
^{Q^{\ast}_{\mu}}}{|\eta^{-1}\xi|^{\mu}}{d}\xi{d}\eta\right)^{\frac{1}{Q^{\ast}_{\mu}}}\longrightarrow0,\qquad \mathrm{as} ~k\rightarrow+\infty.
\end{equation*}
Hence, by Theorem \ref{main thm0}, there exist $(c_k,\lambda_k)\rightarrow (1,+\infty)$ and $\{\xi_k\}\subset \Omega$ such that
\begin{equation}\label{weaknorm1}
  \mathrm{dist}(u_k,\mathfrak{M})=\|u_k-c_k\mathfrak{g}_{\lambda_k,\xi_k}U\|_{S^{1,2}(\mathbb{H}^{n})}
  \rightarrow0,\qquad \mathrm{as} ~k\rightarrow+\infty.
\end{equation}
We claim that there exists $\widetilde{C}>0$ such that
\begin{equation*}
  \inf\limits_{\xi\in \Omega}\|\nabla_\mathbb{H}\mathfrak{g}_{\lambda,\xi}U\|^2_{L^2(\Omega^c)}\geq \widetilde{C}\lambda^{2-Q},\qquad \mathrm{as} ~\lambda\rightarrow+\infty.
\end{equation*}
Indeed, for any $\xi\in \Omega$, there exists a sufficiently large $R>0$ such that
\begin{equation*}
  \|\nabla_\mathbb{H}\mathfrak{g}_{\lambda,\xi}U\|^2_{L^2(\Omega^c)}=\|\nabla_\mathbb{H}\mathfrak{g}_{\lambda,0}U\|^2
  _{L^2((\tau_{\xi^{-1}}(\Omega))^c)}\geq \|\nabla_\mathbb{H}\mathfrak{g}_{\lambda,0}U\|^2_{L^2(B_R^c)}.
\end{equation*}
By direct computations, as $\lambda\rightarrow+\infty$, we get
\begin{align*}
  \|\nabla_\mathbb{H}\mathfrak{g}_{\lambda,0}U\|^2_{L^2(B_R^c)}\geq& C\int_{\delta_\lambda(B_R^c)}
  \frac{|z|^2}{[(1+|z|^2)^2+t^2]^{\frac{Q}{2}}}dz dt
 \geq  C\int_{\delta_\lambda(B_R^c)}
  \frac{|z|^2}{(|z|^4+t^2)^{\frac{Q}{2}}}dz dt\\
  \geq &C\int_{\lambda R}^{+\infty}
  \frac{r^2r^{Q-1}}{r^{2Q}}dr=C\lambda^{2-Q}.
\end{align*}
Hence, we obtain
\begin{equation}\label{weaknorm2}
  \mathrm{dist}(u_k,\mathfrak{M})^2\geq c_k^2 \|\nabla_\mathbb{H}\mathfrak{g}_{\lambda_k,\xi_k}U\|^2_{L^2(\Omega^c)}\geq \widetilde{C}c_k^2 \lambda_k^{2-Q}.
\end{equation}
On the other hand, there holds
\begin{align}\label{weaknorm3}
  \|u_k\|_{L^{\frac{Q}{Q-2}}_w(\Omega)}\leq & \|u_k-c_k\mathfrak{g}_{\lambda_k,\xi_k}U\|_{L^{\frac{Q}{Q-2}}_w(\Omega)}+\|c_k\mathfrak{g}_{\lambda_k,\xi_k}U\|_{L^{\frac{Q}{Q-2}}_w(\Omega)}
 \nonumber \\ \leq & \|u_k-c_k\mathfrak{g}_{\lambda_k,\xi_k}U\|_{L^{\frac{Q}{Q-2}}(\Omega)}+\|c_k\mathfrak{g}_{\lambda_k,\xi_k}U\|_{L^{\frac{Q}{Q-2}}_w(\mathbb{H}^n)} \nonumber\\
  \leq & C \|u_k-c_k\mathfrak{g}_{\lambda_k,\xi_k}U\|_{L^{\frac{2Q}{Q-2}}(\mathbb{H}^n)}+Cc_k\lambda_k^{2-Q}\|U\|_{L^{\frac{Q}{Q-2}}_w(\mathbb{H}^n)} \nonumber\\
  \leq & C  \mathrm{dist}(u_k,\mathfrak{M})+Cc_k\lambda_k^{2-Q}.
\end{align}
Therefore, from \eqref{weaknorm1}-\eqref{weaknorm3}, we deduce that
\begin{equation*}
  \|u_k\|_{L^{\frac{Q}{Q-2}}_w(\Omega)}\leq C  \mathrm{dist}(u_k,\mathfrak{M}).
\end{equation*}
This with Theorem \ref{main thm0}  yields the desired contradiction, and the result follows.
\end{proof}

\section{Proof of Theorem \ref{main thm}}\label{sec5}
In this section, we will prove Theorem \ref{main thm} by adapting the strategy outlined in \cite{CFL25}.
Under the assumptions of Theorem~\ref{main thm}, and following an argument similar to Appendix A in Bahri and Coron \cite{BC88},  the smooth function
\begin{equation*}
G(\tilde{\lambda}_1, \ldots, \tilde{\lambda}_m, \tilde{\xi}_1, \ldots, \tilde{\xi}_m) = \left\| u - \sum_{i=1}^{m} \mathfrak{g}_{\tilde{\lambda}_i, \tilde{\xi}_i} U \right\|_{S^{1,2}(\mathbb{H}^{n})}
\end{equation*}
 attains its minimum at some point \((\lambda_1, \ldots, \lambda_m, \xi_1, \ldots, \xi_m)\) satisfying
\begin{equation*}
\frac{\lambda_i}{\tilde{\lambda}_i} = 1 + o_{\delta}(1),
\quad
\lambda_i \tilde{\lambda}_i |\xi_i^{-1}  \tilde{\xi}_i |^2 = o_{\delta}(1),
\quad \forall \, 1\leq i\leq m.
\end{equation*}
Set
\begin{equation*}
\sigma = \sum_{i=1}^{m} U_i,
\quad
\rho = u - \sigma.
\end{equation*}
By differentiating $G$, we derive the following orthogonality condition:
\begin{equation*}
\left\langle \rho, \frac{\partial \mathfrak{g}_{\lambda_i, \eta}U}{\partial \eta^{(a)}}\bigg|_{\eta=\xi_i} \right\rangle_{S^{1,2}(\mathbb{H}^{n})} = 0,
\quad
\left\langle \rho, \frac{\partial \mathfrak{g}_{r, \xi_i}U}{\partial r}\bigg|_{r=\lambda_i} \right\rangle_{S^{1,2}(\mathbb{H}^{n})} = 0,
\quad \forall\, 1\leq i\leq m,\ 1\leq a\leq 2n+1.
\end{equation*}
It can be verified that the sets
\begin{equation*}
  \left\{ \frac{\partial \mathfrak{g}_{\lambda_i, \eta}U}{\partial \eta^{(a)}}\bigg|_{\eta=\xi_i} \right\}_{1\leq a\leq 2n+1}
\bigcup
\left\{ \frac{\partial \mathfrak{g}_{r, \xi_i}U}{\partial r}\bigg|_{r=\lambda_i} \right\}
\end{equation*}
and
\begin{equation*}
  \{ Z_i^a \}_{1\leq a\leq 2n+2}
\end{equation*}
span the same subspace for each \( 1\leq i\leq m \), and
\begin{equation*}
	-\Delta_{\mathbb{H}} Z_i^a=Q^{\ast}_{\mu}\left(\int_{\mathbb{H}^{n}}\frac{|U_i(\eta)|^{Q^{\ast}_{\mu}-1}Z_i^a(\eta)}
{|\eta^{-1}\xi|^{\mu}}{d}\eta\right)|U_i|^{Q^{\ast}_{\mu}-2}U_i+(Q^{\ast}_{\mu}-1)
\left(\int_{\mathbb{H}^{n}}\frac{|U_i(\eta)|^{Q^{\ast}_{\mu}}}{|\eta^{-1}\xi|^{\mu}}{d}\eta\right)
|U_i|^{Q^{\ast}_{\mu}-2}Z_i^a
\end{equation*}
for any $1\leq i\leq m$ and $1\leq a\leq 2n+2$.
Consequently, we obtain
\begin{align}\label{orth}
Q^{\ast}_{\mu}\int_{\mathbb{H}^{n}}\int_{\mathbb{H}^{n}}\frac{U_i^{Q^{\ast}_{\mu}-1}Z_i^aU_i^{Q^{\ast}_{\mu}-1}\rho}
{|\eta^{-1}\xi|^{\mu}}d \xi{d}\eta +(Q^{\ast}_{\mu}-1)
\int_{\mathbb{H}^{n}}\int_{\mathbb{H}^{n}}\frac{U_i^{Q^{\ast}_{\mu}}U_i^{Q^{\ast}_{\mu}-2}Z_i^a\rho}
{|\eta^{-1}\xi|^{\mu}}d\xi{d}\eta=\langle \rho, Z_i^a \rangle_{S^{1,2}(\mathbb{H}^{n})}= 0
\end{align}
for any $1\leq i\leq m$ and $1\leq a\leq 2n+2$.

For $Q\geq 4$, $0<\mu< Q$ with $\mu\leq 4$. Let
\begin{equation}\label{key}
\begin{aligned}
L(\rho)= f+h +  N(\rho),
\end{aligned}
\end{equation}
where
\begin{equation*}
  L(\rho)=-\Delta_{\mathbb{H} } \rho-Q^{\ast}_{\mu}\left(\int_{\mathbb{H}^{n}}\frac{|\sigma(\eta)|^{Q^{\ast}_{\mu}-1}\rho(\eta)}
{|\eta^{-1}\xi|^{\mu}}{d}\eta\right)|\sigma|^{Q^{\ast}_{\mu}-2}\sigma-(Q^{\ast}_{\mu}-1)
\left(\int_{\mathbb{H}^{n}}\frac{|\sigma(\eta)|^{Q^{\ast}_{\mu}}}{|\eta^{-1}\xi|^{\mu}}{d}\eta\right)
|\sigma|^{Q^{\ast}_{\mu}-2}\rho,
\end{equation*}
\begin{equation*}
f=\left(\int_{\mathbb{H}^{n}}\frac{|\sigma(\eta)|^{Q^{\ast}_{\mu}}}
{|\eta^{-1}\xi|^{\mu}}{d}\eta\right)|\sigma|^{Q^{\ast}_{\mu}-2}\sigma-\sum\limits_{i=1}^m
\left(\int_{\mathbb{H}^{n}}\frac{|U_i(\eta)|^{Q^{\ast}_{\mu}}}
{|\eta^{-1}\xi|^{\mu}}{d}\eta\right)|U_i|^{Q^{\ast}_{\mu}-2}U_i,
\end{equation*}
\begin{equation*}
  h=-\Delta_{\mathbb{H} } u-\left(\int_{\mathbb{H}^{n}}\frac{|u(\eta)|^{Q^{\ast}_{\mu}}}
{|\eta^{-1}\xi|^{\mu}}{d}\eta\right)|u|^{Q^{\ast}_{\mu}-2}u,
\end{equation*}
 and
  \begin{align*}
  N(\rho)=&\left(\int_{\mathbb{H}^{n}}\frac{|u(\eta)|^{Q^{\ast}_{\mu}}}
{|\eta^{-1}\xi|^{\mu}}{d}\eta\right)|u|^{Q^{\ast}_{\mu}-2}u-\left(\int_{\mathbb{H}^{n}}\frac{|\sigma(\eta)|^{Q^{\ast}_{\mu}}}
{|\eta^{-1}\xi|^{\mu}}{d}\eta\right)|\sigma|^{Q^{\ast}_{\mu}-2}\sigma\\
&-Q^{\ast}_{\mu}\left(\int_{\mathbb{H}^{n}}\frac{|\sigma(\eta)|^{Q^{\ast}_{\mu}-1}\rho(\eta)}
{|\eta^{-1}\xi|^{\mu}}{d}\eta\right)|\sigma|^{Q^{\ast}_{\mu}-2}\sigma-(Q^{\ast}_{\mu}-1)
\left(\int_{\mathbb{H}^{n}}\frac{|\sigma(\eta)|^{Q^{\ast}_{\mu}}}{|\eta^{-1}\xi|^{\mu}}{d}\eta\right)
|\sigma|^{Q^{\ast}_{\mu}-2}\rho.
  \end{align*}
  Then we have the following lemmas.
  \begin{lemma}\label{main0}\cite[(2.9)-(2.10), (3.16), Lemma 2.3]{CFL25}
Let $Q\geq4$, $U_i$ and $U_j$ be two bubbles, we have:
 \\
 $\bullet$ For all $i\neq j$,  $\alpha,\beta\geq0$ with $\alpha+\beta=Q^*$, it holds that
\begin{equation}\label{main00}
  \int_{\mathbb{H}^{n}}U_i^\alpha U_j^\beta d\xi\approx
  \begin{cases}
  \varepsilon_{ij}^{\frac{\min\{\alpha,\beta\}(Q-2)}{2}}, \qquad & \mathrm{if} ~ \alpha\neq\beta, \vspace{.2cm}  \\
  \varepsilon_{ij}^{\frac{Q}{2}}|\log \varepsilon_{ij}|, \qquad & \mathrm{if} ~ \alpha=\beta.
  \end{cases}
\end{equation}
$\bullet$ For all $i\neq j$ with $\lambda_i \leq\lambda_j$, it holds that
\begin{equation}\label{main01}
  \int_{\mathbb{H}^{n}}U_j^{Q^*-2} U_i Z_j^{2n+2} d\xi\approx \varepsilon_{ij}^{\frac{Q-2}{2}}.
\end{equation}
$\bullet$ It holds that
\begin{equation}\label{main02}
 \left\| \sigma^{Q^{*}-1}-\sum\limits_{i=1}^m U_i^{Q^*-1}\right\|_{L^{\frac{2Q}{Q+2}}(\mathbb{H}^{n})} \lesssim
 \begin{cases}
 \varepsilon^{\frac{Q+2}{4}}, \qquad & \mathrm{if} ~ Q\geq8, \\
 \varepsilon, \qquad & \mathrm{if} ~ Q=4.
 \end{cases}
\end{equation}
$\bullet$ For all $i\neq j$ with $Q=6$, it holds that
\begin{equation}\label{main03}
 \|U_i U_j\|_{(S^{1,2}(\mathbb{H}^{n}))^{-1}} \lesssim \varepsilon_{ij}^2 |\log \varepsilon_{ij}|^{\frac{1}{2}}.
\end{equation}
$\bullet$ For each $1\leq i \leq m$, it holds that
\begin{equation}\label{main04}
 \left|\int_{\mathbb{H}^{n}}\sigma^{Q^{\ast}-2}\rho Z_i^{2n+2} d \xi\right| \lesssim\|\rho\|_{S^{1,2}(\mathbb{H}^{n})}^2
+o\Big(\varepsilon^{\frac{Q-2}{2}}\Big).
\end{equation}
  \end{lemma}

  \begin{lemma}\label{main1}
 Let $Q\geq 4$, $0<\mu< Q$ with $\mu\leq 4$, then
 \begin{equation*}
 \|N(\rho)\|_{(S^{1,2}(\mathbb{H}^{n}))^{-1}} \lesssim \|\rho\|_{S^{1,2}(\mathbb{H}^{n})}^{\min\{2,Q_\mu^*-1\}}.
 \end{equation*}
\end{lemma}
\begin{proof}
For any $\varphi\in S^{1,2}(\mathbb{H}^{n})$, since $Q_\mu^*\geq2$, by Lemma \ref{gs},  we have
\begin{align*}
 &\langle N(\rho),\varphi\rangle_{S^{1,2}(\mathbb{H}^{n}),(S^{1,2}(\mathbb{H}^{n}))^{-1}} =\int_{\mathbb{H}^{n}}N(\rho)\varphi d\xi \\ =&\int_{\mathbb{H}^{n}}\int_{\mathbb{H}^{n}}\frac{\big[(\sigma+\rho)^{Q_\mu^*}-\sigma^{Q_\mu^*}-Q^*_\mu \sigma^{Q^*_\mu-1}\rho\big](\sigma+\rho)^{Q^*_\mu-1}\varphi}{|\eta^{-1}\xi|^\mu}d\xi d\eta
  \\&+\displaystyle \int_{\mathbb{H}^{n}}\int_{\mathbb{H}^{n}}\frac{\sigma^{Q^*_\mu}\big[(\sigma+\rho)^{Q^*_\mu-1}-\sigma^{Q^*_\mu-1}-(Q^*_\mu-1)\sigma^{Q^*_\mu-2}\rho\big]\varphi}{|\eta^{-1}\xi|^\mu}
  d\xi d\eta\\
  &+Q^*_\mu\displaystyle\int_{\mathbb{H}^{n}}\int_{\mathbb{H}^{n}}\frac{\sigma^{Q^*_\mu-1}\rho
  \big[(\sigma+\rho)^{Q^*_\mu-1}-\sigma^{Q^*_\mu-1}\big]
  \varphi}{|\eta^{-1}\xi|^\mu}d\xi d\eta\\
  \lesssim& \int_{\mathbb{H}^{n}}\int_{\mathbb{H}^{n}}
  \frac{\big( \sigma^{Q^*_\mu-2}\rho^2+\rho^{Q_\mu^*}\big)(\sigma+\rho)^{Q^*_\mu-1}\varphi}{|\eta^{-1}\xi|^\mu}d\xi d\eta
 \\&+ \begin{cases}
             \displaystyle \int_{\mathbb{H}^{n}}\int_{\mathbb{H}^{n}}
             \frac{\sigma^{Q^*_\mu}\big(\sigma^{Q^*_\mu-3}\rho^2+\rho^{Q^*_\mu-1}\big)\varphi}{|\eta^{-1}\xi|^\mu}
  d\xi d\eta, \qquad & \mathrm{if} ~ Q_\mu^*\geq3, \\
              \displaystyle \int_{\mathbb{H}^{n}}\int_{\mathbb{H}^{n}}
             \frac{\sigma^{Q^*_\mu}\rho^{Q^*_\mu-1}\varphi}{|\eta^{-1}\xi|^\mu}
  d\xi d\eta, \qquad& \mathrm{if} ~ 2\leq Q_\mu^*<3,
           \end{cases}.\\
           &+\displaystyle\int_{\mathbb{H}^{n}}\int_{\mathbb{H}^{n}}\frac{\sigma^{Q^*_\mu-1}\rho
  \big(\sigma^{Q^*_\mu-2}\rho+\rho^{Q^*_\mu-1}\big)
  \varphi}{|\eta^{-1}\xi|^\mu}d\xi d\eta.
\end{align*}
Using \eqref{eq:HLSH}, the H\"{o}lder and Sobolev inequalities, we obtain
\begin{align*}
 \big|\langle N(\rho),\varphi\rangle_{S^{1,2}(\mathbb{H}^{n}),(S^{1,2}(\mathbb{H}^{n}))^{-1}} \big|\lesssim \|\rho\|_{S^{1,2}(\mathbb{H}^{n})}^{\min\{2,Q_\mu^*-1\}}\|\varphi\|_{S^{1,2}(\mathbb{H}^{n})}.
\end{align*}
This completes the proof.
\end{proof}

  \begin{lemma}\label{main2}
 Let $Q\geq 4$, $0<\mu< Q$ with $\mu\leq 4$, then
  \begin{equation*}
  \|f\|_{(S^{1,2}(\mathbb{H}^{n}))^{-1}}\lesssim
  \begin{cases}
             \varepsilon^{\frac{Q+2}{4}}, \qquad & \mathrm{if} ~ Q\geq8, \\
             \varepsilon^2|\log \varepsilon|^{\frac{1}{2}},\qquad & \mathrm{if} ~ Q=6,\\
              \varepsilon, \qquad& \mathrm{if} ~ Q=4.
           \end{cases}
  \end{equation*}
\end{lemma}
\begin{proof}
For any $\varphi\in S^{1,2}(\mathbb{H}^{n})$, when $Q\geq8$ or $Q=4$, by \eqref{limit1}, \eqref{eq:HLSH}, \eqref{limit3}, \eqref{main00}, and \eqref{main02}, 
we have
\begin{align}\label{main20}
 &\langle f,\varphi\rangle_{S^{1,2}(\mathbb{H}^{n}),(S^{1,2}(\mathbb{H}^{n}))^{-1}} =\int_{\mathbb{H}^{n}}f\varphi d\xi \nonumber\\ =&\int_{\mathbb{H}^{n}}\int_{\mathbb{H}^{n}}\frac{\sigma^{Q^{\ast}_{\mu}}\sigma^{Q^{\ast}_{\mu}-1}\varphi}
{|\eta^{-1}\xi|^{\mu}}d\xi{d}\eta-\sum\limits_{i=1}^m
\int_{\mathbb{H}^{n}}\int_{\mathbb{H}^{n}}\frac{U_i^{Q^{\ast}_{\mu}}U_i^{Q^{\ast}_{\mu}-1}\varphi}
{|\eta^{-1}\xi|^{\mu}}d\xi{d}\eta \nonumber\\
=&\int_{\mathbb{H}^{n}}\int_{\mathbb{H}^{n}}\frac{\Big[\sigma^{Q^{\ast}_{\mu}}-\sum\limits_{i=1}^mU_i^{Q^{\ast}_{\mu}}\Big]\sigma^{Q^{\ast}_{\mu}-1}\varphi}
{|\eta^{-1}\xi|^{\mu}}d\xi{d}\eta+\sum\limits_{i=1}^m
\int_{\mathbb{H}^{n}}\int_{\mathbb{H}^{n}}\frac{U_i^{Q^{\ast}_{\mu}}\Big[\sigma^{Q^{\ast}_{\mu}-1}-U_i^{Q^{\ast}_{\mu}-1}\Big]\varphi}
{|\eta^{-1}\xi|^{\mu}}d\xi{d}\eta \nonumber\\
\lesssim &\int_{\mathbb{H}^{n}}\int_{\mathbb{H}^{n}}\frac{\Big[\sigma^{Q^{\ast}_{\mu}}-\sum\limits_{i=1}^mU_i^{Q^{\ast}_{\mu}}\Big]\sigma^{Q^{\ast}_{\mu}-1}\varphi}
{|\eta^{-1}\xi|^{\mu}}d\xi{d}\eta+\sum\limits_{i=1}^m
\int_{\mathbb{H}^{n}} U_i^{Q^*-Q^{\ast}_{\mu}}\Big[\sigma^{Q^{\ast}_{\mu}-1}-U_i^{Q^{\ast}_{\mu}-1}\Big]\varphi
d\xi \nonumber\\
\lesssim & \left(\int_{\mathbb{H}^{n}}\bigg[\sigma^{Q^{\ast}_{\mu}}-\sum\limits_{i=1}^mU_i^{Q^{\ast}_{\mu}}\bigg]^{\frac{2Q}{2Q-\mu}}d\xi\right)^{\frac{2Q-\mu}{2Q}}\|\varphi\|_{S^{1,2}(\mathbb{H}^{n})}
+\int_{\mathbb{H}^{n}} \bigg[\sigma^{Q^{*}-1}-\sum\limits_{i=1}^m U_i^{Q^*-1}\bigg]\varphi
d\xi \nonumber\\
\lesssim & \sum\limits_{j\neq i}\left(\int_{\mathbb{H}^{n}}U_i^{\frac{2Q(Q^{\ast}_{\mu}-1)}{2Q-\mu}}U_j
^{\frac{2Q}{2Q-\mu}}d\xi\right)^{\frac{2Q-\mu}{2Q}}\|\varphi\|_{S^{1,2}(\mathbb{H}^{n})}
+\bigg\| \sigma^{Q^{*}-1}-\sum\limits_{i=1}^m U_i^{Q^*-1}\bigg\|_{L^{\frac{2Q}{Q+2}}(\mathbb{H}^{n})}\|\varphi\|_{S^{1,2}(\mathbb{H}^{n})} \nonumber\\
\lesssim &
\begin{cases}
\varepsilon^{\frac{Q-2}{2}}|\log \varepsilon|^{\frac{Q-2}{Q}}\|\varphi\|_{S^{1,2}(\mathbb{H}^{n})} \quad & \mathrm{if} ~ \mu=4,\\
\varepsilon^{\frac{Q-2}{2}}\|\varphi\|_{S^{1,2}(\mathbb{H}^{n})} \quad & \mathrm{if} ~ \mu<4,
\end{cases}
+
\begin{cases}
 \varepsilon^{\frac{Q+2}{4}}\|\varphi\|_{S^{1,2}(\mathbb{H}^{n})}, \quad & \mathrm{if} ~ Q\geq8, \\
 \varepsilon \|\varphi\|_{S^{1,2}(\mathbb{H}^{n})}, \quad & \mathrm{if} ~ Q=4,
 \end{cases} \nonumber\\
 = &
 \begin{cases}
 \varepsilon^{\frac{Q+2}{4}}\|\varphi\|_{S^{1,2}(\mathbb{H}^{n})}, \quad & \mathrm{if} ~ Q\geq8, \\
 \varepsilon \|\varphi\|_{S^{1,2}(\mathbb{H}^{n})}, \quad & \mathrm{if} ~ Q=4.
 \end{cases}
\end{align}
When $Q=6$, by \eqref{limit1} and \eqref{limit3},  we have
\begin{align}\label{main21}
f=&\left(\int_{\mathbb{H}^{n}}\frac{|\sigma(\eta)|^{3-\frac{\mu}{4}}
}{|\eta^{-1}\xi|^{\mu}}{d}\eta\right)|\sigma|^{1-\frac{\mu}{4}}\sigma-\sum\limits_{i=1}^m
\left(\int_{\mathbb{H}^{n}}\frac{|U_i(\eta)|^{3-\frac{\mu}{4}}}
{|\eta^{-1}\xi|^{\mu}}{d}\eta\right)|U_i|^{1-\frac{\mu}{4}}U_i \nonumber\\
=&\left(\int_{\mathbb{H}^{n}}\frac{|\sigma(\eta)|^{3-\frac{\mu}{4}}
-\sum\limits_{i=1}^m|U_i(\eta)|^{3-\frac{\mu}{4}}}{|\eta^{-1}\xi|^{\mu}}{d}\eta\right)|\sigma|^{1-\frac{\mu}{4}}\sigma
+\sum\limits_{i=1}^m
\left(\int_{\mathbb{H}^{n}}\frac{|U_i(\eta)|^{3-\frac{\mu}{4}}}
{|\eta^{-1}\xi|^{\mu}}{d}\eta\right)\Big[|\sigma|^{1-\frac{\mu}{4}}\sigma-|U_i|^{1-\frac{\mu}{4}}U_i\Big]\nonumber\\
\lesssim &\left(\int_{\mathbb{H}^{n}}\frac{|\sigma(\eta)|^{3-\frac{\mu}{4}}
-\sum\limits_{i=1}^m|U_i(\eta)|^{3-\frac{\mu}{4}}}{|\eta^{-1}\xi|^{\mu}}{d}\eta\right)|\sigma|^{1-\frac{\mu}{4}}\sigma
+\sum\limits_{i=1}^m
|U_i|^{\frac{\mu}{4}}\Big[|\sigma|^{1-\frac{\mu}{4}}\sigma-|U_i|^{1-\frac{\mu}{4}}U_i\Big] \nonumber\\
\lesssim  &\left(\int_{\mathbb{H}^{n}}\frac{|\sigma(\eta)|^{3-\frac{\mu}{4}}
-\sum\limits_{i=1}^m|U_i(\eta)|^{3-\frac{\mu}{4}}}{|\eta^{-1}\xi|^{\mu}}{d}\eta\right)|\sigma|^{1-\frac{\mu}{4}}\sigma
+
\sigma^2-\sum\limits_{i=1}^mU_i^2 \nonumber\\
=  &\underbrace{\left(\int_{\mathbb{H}^{n}}\frac{|\sigma(\eta)|^{3-\frac{\mu}{4}}
-\sum\limits_{i=1}^m|U_i(\eta)|^{3-\frac{\mu}{4}}}{|\eta^{-1}\xi|^{\mu}}{d}\eta\right)|\sigma|^{1-\frac{\mu}{4}}\sigma
}_{:=f_1}+\underbrace{\sum\limits_{i\neq j} U_i U_j}_{:=f_2}.
\end{align}
By \eqref{main03}, we obtain
\begin{equation}\label{main22}
  \|f_2\|_{(S^{1,2}(\mathbb{H}^{n}))^{-1}}\lesssim \varepsilon^2|\log \varepsilon|^{\frac{1}{2}}.
\end{equation}
Similarly, when $\mu<4$, for any $\varphi\in S^{1,2}(\mathbb{H}^{n})$, using \eqref{eq:HLSH} and \eqref{main00} again, there holds
\begin{equation}\label{main23}
  \big|\langle f_1,\varphi\rangle_{S^{1,2}(\mathbb{H}^{n}),(S^{1,2}(\mathbb{H}^{n}))^{-1}}\big|\lesssim \varepsilon^{2}\|\varphi\|_{S^{1,2}(\mathbb{H}^{n})}\lesssim \varepsilon^{2}|\log \varepsilon|^{\frac{1}{2}}\|\varphi\|_{S^{1,2}(\mathbb{H}^{n})}.
\end{equation}
Moreover, when $\mu=4$,
\begin{align}\label{main24}
  f_1=&\left(\int_{\mathbb{H}^{n}}\frac{|\sigma(\eta)|^{2}
-\sum\limits_{i=1}^m|U_i(\eta)|^{2}}{|\eta^{-1}\xi|^{4}}{d}\eta\right)\sigma
=\sum\limits_{i=1}^m\underbrace{\left(\int_{\mathbb{H}^{n}}\frac{|\sigma(\eta)|^{2}
-\sum\limits_{i=1}^m|U_i(\eta)|^{2}}{|\eta^{-1}\xi|^{4}}{d}\eta\right)U_i}_{:=f_{1i}}.
\end{align}
For each $1\leq i\leq m$, by \eqref{limit1} and \eqref{limit3},
\begin{align}\label{main25}
  f_{1i}=&\sum\limits_{j\neq i}\left(\int_{\mathbb{H}^{n}}\frac{U_i (\eta) U_j(\eta)}{|\eta^{-1}\xi|^{4}}{d}\eta\right)U_i+\sum\limits_{j,l\neq i}\left(\int_{\mathbb{H}^{n}}\frac{U_j(\eta) U_l(\eta)}{|\eta^{-1}\xi|^{4}}{d}\eta\right)U_i \nonumber\\
  \lesssim &\sum\limits_{j\neq i}\left(\int_{\mathbb{H}^{n}}\frac{U_i (\eta) U_j(\eta)}{|\eta^{-1}\xi|^{4}}{d}\eta\right)U_i+\sum\limits_{j\neq i}\left(\int_{\mathbb{H}^{n}}\frac{U_j^2(\eta)}{|\eta^{-1}\xi|^{4}}{d}\eta\right)U_i+\sum\limits_{l\neq i}\left(\int_{\mathbb{H}^{n}}\frac{U_l^2(\eta)}{|\eta^{-1}\xi|^{4}}{d}\eta\right)U_i \nonumber\\
   \lesssim&\sum\limits_{j\neq i}\underbrace{\left(\int_{\mathbb{H}^{n}}\frac{U_i(\eta) U_j(\eta)}{|\eta^{-1}\xi|^{4}}{d}\eta\right)U_i}_{:=f_{ij}}+\underbrace{\sum\limits_{j\neq i}U_iU_j}_{=f_2}.
\end{align}
Let
\begin{equation*}
  \omega_1(\xi)=\int_{\mathbb{H}^{n}}\frac{U_i(\eta) U_j(\eta)}{|\eta^{-1}\xi|^{4}}{d}\eta,
\end{equation*}
and $\omega_2$ satisfy
\begin{equation*}
 -\Delta_{\mathbb{H}} \omega_2=f_{ij},
\end{equation*}
then
\begin{equation*}
  -\Delta_{\mathbb{H}}\omega_1=U_i U_j,\quad -\Delta_{\mathbb{H}}\omega_2=\omega_1U_i.
\end{equation*}
Using the H\"{o}lder and Sobolev inequalities, we get
\begin{align*}
  \|f_{ij}\|^2_{(S^{1,2}(\mathbb{H}^{n}))^{-1}}=\|\omega_2\|^2_{S^{1,2}(\mathbb{H}^{n})}=
  \int_{\mathbb{H}^{n}}\omega_1\omega_2U_id\xi\lesssim & \|\omega_1\|_{S^{1,2}(\mathbb{H}^{n})}\|\omega_2\|_{S^{1,2}(\mathbb{H}^{n})}\\
  =& \|U_i U_j\|_{(S^{1,2}(\mathbb{H}^{n}))^{-1}}\|\omega_2\|_{(S^{1,2}(\mathbb{H}^{n}))}.
\end{align*}
Therefore, for $j\neq i$,
\begin{equation}\label{main26}
  \|f_{ij}\|_{(S^{1,2}(\mathbb{H}^{n}))^{-1}}=\|\omega_2\|_{S^{1,2}(\mathbb{H}^{n})}\lesssim \|U_i U_j\|_{(S^{1,2}(\mathbb{H}^{n}))^{-1}}\lesssim \varepsilon^2|\log \varepsilon|^{\frac{1}{2}}.
\end{equation}
From \eqref{main20}-\eqref{main26}, the conclusion is reached.
\end{proof}

\begin{lemma}\label{main3}
Let $Q=4$ and $\mu\in (Q-2,4)$. For fixed $k$, we have
\begin{align*}
\int_{\mathbb{H}^n } f Z_k^{2n+2}  d\xi=&Q^{\ast}_{\mu} \sum_{i\neq k}\int_{\mathbb{H}^{n}}\int_{\mathbb{H}^{n}}
\frac{U_k^{Q^{\ast}_{\mu}-1}U_iU_k^{Q^{\ast}_{\mu}-1}Z_k^{2n+2}}
{|\eta^{-1}\xi|^{\mu}}d \xi{d}\eta \\
&+(Q^{\ast}_{\mu}-1)\sum_{i\neq k}
\int_{\mathbb{H}^{n}}\int_{\mathbb{H}^{n}}\frac{U_k^{Q^{\ast}_{\mu}}U_k^{Q^{\ast}_{\mu}-2}U_iZ_k^{2n+2}}
{|\eta^{-1}\xi|^{\mu}}d\xi{d}\eta
+o\Big(\varepsilon^{\frac{Q-2}{2}}\Big).
\end{align*}
\end{lemma}
\begin{proof}
Denote $\sigma_{[k]}=\sum\limits_{i\neq k}U_i$. Then
\begin{align}\label{main30}
  &\int_{\mathbb{H}^{n}}fZ_k^{2n+2} d\xi \nonumber  \\=&\int_{\mathbb{H}^{n}}\int_{\mathbb{H}^{n}}\frac{\sigma^{Q^{\ast}_{\mu}}\sigma^{Q^{\ast}_{\mu}-1}Z_k^{2n+2}}
{|\eta^{-1}\xi|^{\mu}}d\xi{d}\eta-\sum\limits_{i=1}^m
\int_{\mathbb{H}^{n}}\int_{\mathbb{H}^{n}}\frac{U_i^{Q^{\ast}_{\mu}}U_i^{Q^{\ast}_{\mu}-1}Z_k^{2n+2}}
{|\eta^{-1}\xi|^{\mu}}d\xi{d}\eta \nonumber\\
=&\underbrace{\int_{\mathbb{H}^{n}}\int_{\mathbb{H}^{n}}\frac{\Big[\sigma^{Q^{\ast}_{\mu}}-\sum\limits_{i=1}^mU_i^{Q^{\ast}_{\mu}}\Big]\sigma^{Q^{\ast}_{\mu}-1}Z_k^{2n+2}}
{|\eta^{-1}\xi|^{\mu}}d\xi{d}\eta}_{:=I}+\underbrace{\sum\limits_{i=1}^m
\int_{\mathbb{H}^{n}}\int_{\mathbb{H}^{n}}\frac{U_i^{Q^{\ast}_{\mu}}\Big[\sigma^{Q^{\ast}_{\mu}-1}-U_i^{Q^{\ast}_{\mu}-1}\Big]Z_k^{2n+2}}
{|\eta^{-1}\xi|^{\mu}}d\xi{d}\eta}_{:=II}.
\end{align}
For $I$, we have
\begin{align}\label{main31}
  I=&\underbrace{\int_{\mathbb{H}^{n}}\int_{\mathbb{H}^{n}}\frac{
  \Big[\sigma^{Q^{\ast}_{\mu}}-\sum\limits_{i=1}^mU_i^{Q^{\ast}_{\mu}}-Q_\mu^*U_k^{Q^{\ast}_{\mu}-1}\sigma_{[k]}\Big]\sigma^{Q^{\ast}_{\mu}-1}Z_k^{2n+2}}
{|\eta^{-1}\xi|^{\mu}}d\xi{d}\eta}_{:=I_1} \nonumber\\
&+Q_\mu^*\underbrace{\int_{\mathbb{H}^{n}}\int_{\mathbb{H}^{n}}\frac{U_k^{Q^{\ast}_{\mu}-1}\sigma_{[k]}\Big[\sigma^{Q^{\ast}_{\mu}-1}-U_k^{Q^{\ast}_{\mu}-1}\Big]Z_k^{2n+2}}{|\eta^{-1}\xi|^{\mu}}d\xi{d}\eta}_{:=I_2}
\nonumber\\
&+Q_\mu^*\int_{\mathbb{H}^{n}}\int_{\mathbb{H}^{n}}\frac{U_k^{Q^{\ast}_{\mu}-1}\sigma_{[k]}U_k^{Q^{\ast}_{\mu}-1}Z_k^{2n+2}}{|\eta^{-1}\xi|^{\mu}}d\xi{d}\eta.
\end{align}
For $II$,  we have
\begin{align}\label{main32}
  II=&\sum\limits_{i\neq k}\int_{\mathbb{H}^{n}}\int_{\mathbb{H}^{n}}\frac{U_i^{Q^{\ast}_{\mu}}\Big[\sigma^{Q^{\ast}_{\mu}-1}-U_i^{Q^{\ast}_{\mu}-1}\Big]Z_k^{2n+2}}
{|\eta^{-1}\xi|^{\mu}}d\xi{d}\eta+\int_{\mathbb{H}^{n}}\int_{\mathbb{H}^{n}}\frac{U_k^{Q^{\ast}_{\mu}}\Big[\sigma^{Q^{\ast}_{\mu}-1}-U_k^{Q^{\ast}_{\mu}-1}\Big]Z_k^{2n+2}}
{|\eta^{-1}\xi|^{\mu}}d\xi{d}\eta \nonumber\\
=&\underbrace{\sum\limits_{i\neq k}\int_{\mathbb{H}^{n}}\int_{\mathbb{H}^{n}}\frac{U_i^{Q^{\ast}_{\mu}}\Big[\sigma^{Q^{\ast}_{\mu}-1}-U_i^{Q^{\ast}_{\mu}-1}\Big]Z_k^{2n+2}}
{|\eta^{-1}\xi|^{\mu}}d\xi{d}\eta}_{:=II_{1}} \nonumber\\
&+\underbrace{\int_{\mathbb{H}^{n}}\int_{\mathbb{H}^{n}}\frac{U_k^{Q^{\ast}_{\mu}}
\Big[\sigma^{Q^{\ast}_{\mu}-1}-U_k^{Q^{\ast}_{\mu}-1}-(Q^{\ast}_{\mu}-1)U_k^{Q^{\ast}_{\mu}-2}\sigma_{[k]}\Big]Z_k^{2n+2}}
{|\eta^{-1}\xi|^{\mu}}d\xi{d}\eta}_{:=II_2} \nonumber\\
&+(Q^{\ast}_{\mu}-1)\int_{\mathbb{H}^{n}}\int_{\mathbb{H}^{n}}
\frac{U_k^{Q^{\ast}_{\mu}}U_k^{Q^{\ast}_{\mu}-2}\sigma_{[k]}Z_k^{2n+2}}
{|\eta^{-1}\xi|^{\mu}}d\xi{d}\eta.
\end{align}
In the following, we estimate $I_1$, $I_2$ and $II_1$, $II_2$. Firstly, by \eqref{limit1}, \eqref{limit3}, \eqref{main00}, and $\big|Z_k^{2n+2}\big| \lesssim U_k$, for some $\theta>0$ small enough, we have
\begin{align}\label{main33}
  |I_1|\lesssim& \int_{\mathbb{H}^{n}}\int_{\mathbb{H}^{n}}\frac{
  \Big[U_k^{Q^{\ast}_{\mu}-2}\sigma^2_{[k]}\chi_{\{U_k\geq \sigma_{[k]}\}}+\sum\limits_{i\neq k}U_i^{Q^{\ast}_{\mu}-1}\sigma_{[i]}\chi_{\{U_k< \sigma_{[k]}\}}\Big]\sigma^{Q^{\ast}_{\mu}-1}Z_k^{2n+2}}
{|\eta^{-1}\xi|^{\mu}}d\xi{d}\eta \nonumber\\
\lesssim& \int_{\mathbb{H}^{n}}\int_{\mathbb{H}^{n}}\frac{
U_k^{Q^{\ast}_{\mu}}U_k^{Q^{\ast}_{\mu}-2}\sigma^2_{[k]}\chi_{\{U_k\geq \sigma_{[k]}\}}}
{|\eta^{-1}\xi|^{\mu}}d\xi{d}\eta+\sum\limits_{i\neq k}\int_{\mathbb{H}^{n}}\int_{\mathbb{H}^{n}}\frac{
  U_k^{Q^{\ast}_{\mu}}U_i^{Q^{\ast}_{\mu}-1}\sigma_{[i]}
  }
{|\eta^{-1}\xi|^{\mu}}d\xi{d}\eta \nonumber\\
\lesssim& \int_{\mathbb{H}^{n}}U_k^{Q^{\ast}-2}\sigma^2_{[k]}\chi_{\{U_k\geq \sigma_{[k]}\}}d\xi +\sum\limits_{i\neq k}\int_{\mathbb{H}^{n}}U_k^{Q^{\ast}-Q_\mu^*}U_i^{Q^{\ast}_{\mu}-1}\sigma_{[i]}d\xi \nonumber\\
\lesssim& \int_{\mathbb{H}^{n}}U_k^{Q^{\ast}-1-\theta}\sigma^{1+\theta}_{[k]} d\xi +
\int_{\mathbb{H}^{n}}\sigma_{[i]}^{Q^{\ast}-Q_\mu^*+1}U_i^{Q^{\ast}_{\mu}-1}d\xi=o\Big(\varepsilon^{\frac{Q-2}{2}}\Big).
\end{align}
Using Lemma \ref{gs}, \eqref{eq:HLSH} and \eqref{main00}, we deduce that
\begin{align}\label{main34}
 |I_2| \lesssim & \int_{\mathbb{H}^{n}}\int_{\mathbb{H}^{n}}\frac{U_k^{Q^{\ast}_{\mu}-1}\sigma_{[k]}\Big[U_k^{Q^{\ast}_{\mu}-2}\sigma_{[k]}
+\sigma_{[k]}^{Q^{\ast}_{\mu}-1}\Big]Z_k^{2n+2}}{|\eta^{-1}\xi|^{\mu}}d\xi{d}\eta \nonumber\\
\lesssim&\int_{\mathbb{H}^{n}}\int_{\mathbb{H}^{n}}\frac{U_k^{Q^{\ast}_{\mu}-1}\sigma_{[k]}U_k^{Q^{\ast}_{\mu}-1}\sigma_{[k]}
}{|\eta^{-1}\xi|^{\mu}}d\xi{d}\eta+\int_{\mathbb{H}^{n}}\int_{\mathbb{H}^{n}}\frac{U_k^{Q^{\ast}_{\mu}-1}\sigma_{[k]}
\sigma_{[k]}^{Q^{\ast}_{\mu}-1}U_k}{|\eta^{-1}\xi|^{\mu}}d\xi{d}\eta \nonumber\\
\lesssim&\left(\int_{\mathbb{H}^{n}}\bigg(U_k^{Q^{\ast}_{\mu}-1}\sigma_{[k]}\bigg)^{\frac{2Q}{2Q-\mu}} d\xi\right)^{\frac{2Q-\mu}{Q}} \nonumber\\
&+\left(\int_{\mathbb{H}^{n}}\bigg(U_k^{Q^{\ast}_{\mu}-1}\sigma_{[k]}\bigg)^{\frac{2Q}{2Q-\mu}} d\xi\right)^{\frac{2Q-\mu}{2Q}}\left(\int_{\mathbb{H}^{n}}\bigg(\sigma_{[k]}^{Q^{\ast}_{\mu}-1}U_k\bigg)^{\frac{2Q}{2Q-\mu}} d\xi\right)^{\frac{2Q-\mu}{2Q}} \nonumber\\
=&O\big(\varepsilon^{Q-2}\big)=o\Big(\varepsilon^{\frac{Q-2}{2}}\Big).
\end{align}
Similarly,
\begin{align}\label{main35}
 |II_1| \lesssim & \sum\limits_{i\neq k} \int_{\mathbb{H}^{n}}\int_{\mathbb{H}^{n}}\frac{U_i^{Q^{\ast}_{\mu}}\Big[U_i^{Q^{\ast}_{\mu}-2}\sigma_{[i]}
\chi_{\{U_i\geq \sigma_{[i]}\}}
+\sigma_{[i]}^{Q^{\ast}_{\mu}-1}\chi_{\{U_i< \sigma_{[i]}\}}\Big]Z_k^{2n+2}}{|\eta^{-1}\xi|^{\mu}}d\xi{d}\eta \nonumber\\
\lesssim& \sum\limits_{i\neq k}\int_{\mathbb{H}^{n}}\int_{\mathbb{H}^{n}}\frac{U_i^{Q^{\ast}_{\mu}}U_i^{Q^{\ast}_{\mu}-2}U_k\sigma_{[i]}
\chi_{\{U_i\geq \sigma_{[i]}\}}
}{|\eta^{-1}\xi|^{\mu}}d\xi{d}\eta+\sum\limits_{i\neq k}\int_{\mathbb{H}^{n}}\int_{\mathbb{H}^{n}}\frac{U_i^{Q^{\ast}_{\mu}}
\sigma_{[i]}^{Q^{\ast}_{\mu}-1}U_k}{|\eta^{-1}\xi|^{\mu}}d\xi{d}\eta \nonumber\\
\lesssim&\sum\limits_{i\neq k} \int_{\mathbb{H}^{n}}U_i^{Q^{\ast}-2}U_k\sigma_{[i]}\chi_{\{U_i\geq \sigma_{[i]}\}}d\xi +\sum\limits_{i\neq k}\int_{\mathbb{H}^{n}}U_i^{Q^{\ast}-Q_\mu^*}\sigma_{[i]}^{Q^{\ast}_{\mu}-1}U_k
d\xi \nonumber\\
\lesssim&
\int_{\mathbb{H}^{n}}U_i^{Q^{\ast}-1-\theta}\sigma^{1+\theta}_{[i]}d\xi +
\int_{\mathbb{H}^{n}}U_i^{Q^{\ast}-Q_\mu^*}\sigma_{[i]}^{Q^{\ast}_{\mu}}d\xi \nonumber\\
=&o\Big(\varepsilon^{\frac{Q-2}{2}}\Big)+O\big(\varepsilon^{\frac{\mu}{2}}\big)=o\Big(\varepsilon^{\frac{Q-2}{2}}\Big),
\end{align}
and
\begin{align}\label{main36}
  |II_2| \lesssim&
  \int_{\mathbb{H}^{n}}\int_{\mathbb{H}^{n}}\frac{U_k^{Q^{\ast}_{\mu}}
\sigma_{[k]}^{Q^{\ast}_{\mu}-1}Z_k^{2n+2}}
{|\eta^{-1}\xi|^{\mu}}d\xi{d}\eta\lesssim
  \int_{\mathbb{H}^{n}}\int_{\mathbb{H}^{n}}\frac{U_k^{Q^{\ast}_{\mu}}U_k
\sigma_{[k]}^{Q^{\ast}_{\mu}-1}}
{|\eta^{-1}\xi|^{\mu}}d\xi{d}\eta
\nonumber\\ \lesssim&\int_{\mathbb{H}^{n}}U_k^{Q^{\ast}-Q_\mu^*+1}\sigma_{[k]}^{Q^{\ast}_{\mu}-1}d\xi=o\Big(\varepsilon^{\frac{Q-2}{2}}\Big),
\end{align}
where we have used the fact that $\mu>Q-2$. Combining \eqref{main30}-\eqref{main36}, we complete the proof.
\end{proof}
Thus, combining Lemma \ref{main3} with \eqref{limit1}, \eqref{limit3}, we have
\begin{equation}\label{main37}
  \int_{\mathbb{H}^n } f Z_k^{2n+2}  d\xi\approx \sum_{i\neq k}\int U_k^{Q^*-2}U_i Z_k^{2n+2} d\xi +o\Big(\varepsilon^{\frac{Q-2}{2}}\Big).
\end{equation}

\begin{lemma}\label{main4}
Let $Q=4$ and $\mu\in (Q-2,4)$, then
\begin{equation*}
     \varepsilon^{\frac{Q-2}{2}} \lesssim  \|\rho\|_{S^{1,2}(\mathbb{H}^{n})}^{\min\{2,Q_\mu^*-1\}}
     +\|h\|_{(S^{1,2}(\mathbb{H}^{n}))^{-1}}.
   \end{equation*}
\end{lemma}
\begin{proof}
Testing the equation \eqref{key} with $Z_i^{2n+2}$, by \eqref{orth}, we obtain
\begin{align}\label{main40}
&\left| \int_{\mathbb{H}^n} f Z_i^{2n+2} d\xi\right| \nonumber\\
\lesssim & \left|Q^{\ast}_{\mu}\int_{\mathbb{H}^{n}}\int_{\mathbb{H}^{n}}
\frac{\sigma^{Q^{\ast}_{\mu}-1}\rho \sigma^{Q^{\ast}_{\mu}-1} Z_i^{2n+2}}
{|\eta^{-1}\xi|^{\mu}} d\xi{d}\eta+(Q^{\ast}_{\mu}-1)
\int_{\mathbb{H}^{n}}\int_{\mathbb{H}^{n}}\frac{\sigma^{Q^{\ast}_{\mu}}
\sigma^{Q^{\ast}_{\mu}-2}\rho Z_i^{2n+2}}{|\eta^{-1}\xi|^{\mu}}{d}\xi d\eta
\right|
\nonumber\\&
+ \|h\|_{(S^{1,2}(\mathbb{H}^{n}))^{-1}}
+ \left|\int_{\mathbb{H}^n} N(\rho) Z_i^{2n+2}  d\xi \right| \nonumber\\
\lesssim&\left|{\int_{\mathbb{H}^{n}}\int_{\mathbb{H}^{n}}
\frac{\sigma^{Q^{\ast}_{\mu}-1}\rho \sigma^{Q^{\ast}_{\mu}-1} Z_i^{2n+2}-U_i^{Q^{\ast}_{\mu}-1}\rho U_i^{Q^{\ast}_{\mu}-1} Z_i^{2n+2}}
{|\eta^{-1}\xi|^{\mu}} d\xi{d}\eta}\right| \nonumber\\
&+\left|{
\int_{\mathbb{H}^{n}}\int_{\mathbb{H}^{n}}\frac{\sigma^{Q^{\ast}_{\mu}}
\sigma^{Q^{\ast}_{\mu}-2}\rho Z_i^{2n+2}-U_i^{Q^{\ast}_{\mu}}
U_i^{Q^{\ast}_{\mu}-2}\rho Z_i^{2n+2}}{|\eta^{-1}\xi|^{\mu}}{d}\xi d\eta
}\right|
\nonumber\\&
+ \|h\|_{(S^{1,2}(\mathbb{H}^{n}))^{-1}}
+ {\left|\int_{\mathbb{H}^n} N(\rho) Z_i^{2n+2}  d\xi \right|} \nonumber\\
=&\left|\underbrace{\int_{\mathbb{H}^{n}}\int_{\mathbb{H}^{n}}
\frac{\Big[\sigma^{Q^{\ast}_{\mu}-1}-U_i^{Q^{\ast}_{\mu}-1}\Big]\rho \sigma^{Q^{\ast}_{\mu}-1} Z_i^{2n+2}}
{|\eta^{-1}\xi|^{\mu}} d\xi{d}\eta}_{:=I}+\underbrace{\int_{\mathbb{H}^{n}}\int_{\mathbb{H}^{n}}
\frac{U_i^{Q^{\ast}_{\mu}-1}\rho \Big[\sigma^{Q^{\ast}_{\mu}-1}-U_i^{Q^{\ast}_{\mu}-1}\Big] Z_i^{2n+2}}
{|\eta^{-1}\xi|^{\mu}} d\xi{d}\eta}_{:=II}\right|\nonumber\\
&+\left|\underbrace{\int_{\mathbb{H}^{n}}\int_{\mathbb{H}^{n}}
\frac{\Big[\sigma^{Q^{\ast}_{\mu}}-U_i^{Q^{\ast}_{\mu}}\Big] \sigma^{Q^{\ast}_{\mu}-2}\rho Z_i^{2n+2}}
{|\eta^{-1}\xi|^{\mu}} d\xi{d}\eta}_{:=III}+\underbrace{\int_{\mathbb{H}^{n}}\int_{\mathbb{H}^{n}}
\frac{U_i^{Q^{\ast}_{\mu}} \Big[\sigma^{Q^{\ast}_{\mu}-2}-U_i^{Q^{\ast}_{\mu}-2}\Big]\rho Z_i^{2n+2}}
{|\eta^{-1}\xi|^{\mu}} d\xi{d}\eta}_{:=IV}\right|\nonumber\\
&+ \|h\|_{(S^{1,2}(\mathbb{H}^{n}))^{-1}}
+ \left|\underbrace{\int_{\mathbb{H}^n} N(\rho) Z_i^{2n+2}  d\xi }_{:=V}\right|.
\end{align}
By Lemma \ref{main1} and $\big|Z_i^{2n+2}\big| \lesssim U_i$, we obtain
\begin{equation}\label{main41}
  |V|\lesssim \|\rho\|_{S^{1,2}(\mathbb{H}^{n})}^{\min\{2,Q_\mu^*-1\}}.
\end{equation}
In the following, we estimate $I$-$IV$.
 For $I$, by Lemma \ref{gs}, \eqref{limit1}, \eqref{limit3}, and \eqref{main00}, for some $\theta>0$ small enough, we have
\begin{align}\label{main42}
  |I|\lesssim&\int_{\mathbb{H}^{n}}\int_{\mathbb{H}^{n}}
\frac{\Big[U_i^{Q^{\ast}_{\mu}-2}\sigma_{[i]}
\chi_{\{U_i\geq \sigma_{[i]}\}}+\sigma_{[i]}^{Q^{\ast}_{\mu}-1}\chi_{\{U_i< \sigma_{[i]}\}}\Big]\rho \sigma^{Q^{\ast}_{\mu}-1} Z_i^{2n+2}}
{|\eta^{-1}\xi|^{\mu}} d\xi{d}\eta \nonumber\\
\lesssim&\int_{\mathbb{H}^{n}}\int_{\mathbb{H}^{n}}
\frac{U_i^{Q^{\ast}_{\mu}} U_i^{Q^{\ast}_{\mu}-2}\sigma_{[i]}\rho
\chi_{\{U_i\geq \sigma_{[i]}\}} }
{|\eta^{-1}\xi|^{\mu}} d\xi{d}\eta+\int_{\mathbb{H}^{n}}\int_{\mathbb{H}^{n}}
\frac{U_i^{Q^{\ast}_{\mu}} \sigma_{[i]}^{Q^{\ast}_{\mu}-1}\rho
\chi_{\{U_i< \sigma_{[i]}\}} }
{|\eta^{-1}\xi|^{\mu}} d\xi{d}\eta \nonumber\\
\lesssim&\int_{\mathbb{H}^{n}}U_i^{Q^{\ast}-2}\sigma_{[i]}\rho
\chi_{\{U_i\geq \sigma_{[i]}\}}d\xi+\int_{\mathbb{H}^{n}}U_i^{Q^*-Q_\mu^{\ast}}\sigma_{[i]}^{Q^{\ast}_{\mu}-1}\rho
\chi_{\{U_i< \sigma_{[i]}\}}d\xi \nonumber\\
=&\int_{\{|\rho|\geq \sigma_{[i]}\}}U_i^{Q^{\ast}-2}\sigma_{[i]}\rho
\chi_{\{U_i\geq \sigma_{[i]}\}}d\xi+\int_{\{|\rho|< \sigma_{[i]}\}}U_i^{Q^{\ast}-2}\sigma_{[i]}\rho
\chi_{\{U_i\geq \sigma_{[i]}\}}d\xi \nonumber\\
&+\int_{\{|\rho|\geq \sigma_{[i]}\}}U_i^{Q^*-Q_\mu^{\ast}}\sigma_{[i]}^{Q^{\ast}_{\mu}-1}\rho
\chi_{\{U_i< \sigma_{[i]}\}}d\xi+\int_{\{|\rho|< \sigma_{[i]}\}}U_i^{Q^*-Q_\mu^{\ast}}\sigma_{[i]}^{Q^{\ast}_{\mu}-1}\rho
\chi_{\{U_i< \sigma_{[i]}\}}d\xi \nonumber\\
\lesssim&\int_{\mathbb{H}^{n}}U_i^{Q^{\ast}-2}\rho^2
d\xi+\int_{\mathbb{H}^{n}}U_i^{Q^{\ast}-1-\theta}\sigma^{1+\theta}_{[i]}
d\xi+\int_{\mathbb{H}^{n}}U_i^{Q^*-Q_\mu^{\ast}}\rho^{Q^{\ast}_{\mu}}
d\xi+\int_{\mathbb{H}^{n}}U_i^{Q^*-Q_\mu^{\ast}}\sigma_{[i]}^{Q^{\ast}_{\mu}}
d\xi \nonumber\\
\lesssim& \|\rho\|_{S^{1,2}(\mathbb{H}^{n})}^2
+o\Big(\varepsilon^{\frac{Q-2}{2}}\Big)+O\big(\varepsilon^{\frac{\mu}{2}}\big)=\|\rho\|_{S^{1,2}(\mathbb{H}^{n})}^2
+o\Big(\varepsilon^{\frac{Q-2}{2}}\Big),
\end{align}
where we have used the fact that $\mu>Q-2$.
For $II$, by \eqref{eq:HLSH} and \eqref{main00}, we have
\begin{align}\label{main43}
  |II|
\lesssim&\int_{\mathbb{H}^{n}}\int_{\mathbb{H}^{n}}
\frac{ U_i^{Q^{\ast}_{\mu}-1}\rho\Big[U_i^{Q^{\ast}_{\mu}-2}\sigma_{[i]}
\chi_{\{U_i\geq \sigma_{[i]}\}}+\sigma_{[i]}^{Q^{\ast}_{\mu}-1}\chi_{\{U_i< \sigma_{[i]}\}}\Big] Z_i^{2n+2}}
{|\eta^{-1}\xi|^{\mu}} d\xi{d}\eta \nonumber\\
\lesssim&\int_{\mathbb{H}^{n}}\int_{\mathbb{H}^{n}}
\frac{U_i^{Q^{\ast}_{\mu}-1}\rho U_i^{Q^{\ast}_{\mu}-1}\sigma_{[i]}
}
{|\eta^{-1}\xi|^{\mu}} d\xi{d}\eta+\int_{\mathbb{H}^{n}}\int_{\mathbb{H}^{n}}
\frac{U_i^{Q^{\ast}_{\mu}-1}\rho \sigma_{[i]}^{Q^{\ast}_{\mu}-1}U_i
}
{|\eta^{-1}\xi|^{\mu}} d\xi{d}\eta \nonumber\\
\lesssim&\left(\int_{\mathbb{H}^{n}}\bigg(U_i^{Q^{\ast}_{\mu}-1}\rho\bigg)^{\frac{2Q}{2Q-\mu}} d\xi\right)^{\frac{2Q-\mu}{2Q}}\left(\int_{\mathbb{H}^{n}}\bigg(U_i^{Q^{\ast}_{\mu}-1}\sigma_{[i]}\bigg)^{\frac{2Q}{2Q-\mu}} 
 d\xi\right)^{\frac{2Q-\mu}{2Q}} \nonumber\\
&+\left(\int_{\mathbb{H}^{n}}\bigg(U_i^{Q^{\ast}_{\mu}-1}\rho\bigg)^{\frac{2Q}{2Q-\mu}} d\xi\right)^{\frac{2Q-\mu}{2Q}}\left(\int_{\mathbb{H}^{n}}\bigg(\sigma_{[i]}^{Q^{\ast}_{\mu}-1}U_i\bigg)^{\frac{2Q}{2Q-\mu}} 
d\xi\right)^{\frac{2Q-\mu}{2Q}} \nonumber\\
\lesssim& \varepsilon^{\frac{Q-2}{2}}\|\rho\|_{S^{1,2}(\mathbb{H}^{n})}   \lesssim \|\rho\|^2_{S^{1,2}(\mathbb{H}^{n})}+\varepsilon^{{Q-2}}=\|\rho\|^2_{S^{1,2}(\mathbb{H}^{n})}
+o\Big(\varepsilon^{\frac{Q-2}{2}}\Big).
\end{align}
Similarly,
\begin{align}\label{main44}
  |III|\lesssim&\int_{\mathbb{H}^{n}}\int_{\mathbb{H}^{n}}
\frac{\Big[U_i^{Q^{\ast}_{\mu}-1}\sigma_{[i]}
\chi_{\{U_i\geq \sigma_{[i]}\}}+\sigma_{[i]}^{Q^{\ast}_{\mu}}\chi_{\{U_i< \sigma_{[i]}\}}\Big] \sigma^{Q^{\ast}_{\mu}-2} \rho Z_i^{2n+2}}
{|\eta^{-1}\xi|^{\mu}} d\xi{d}\eta \nonumber\\
\lesssim&\int_{\mathbb{H}^{n}}\int_{\mathbb{H}^{n}}
\frac{\sigma^{Q^{\ast}_{\mu}-1}\rho U_i^{Q^{\ast}_{\mu}-1}\sigma_{[i]}
}
{|\eta^{-1}\xi|^{\mu}} d\xi{d}\eta+\int_{\mathbb{H}^{n}}\int_{\mathbb{H}^{n}}
\frac{ \sigma_{[i]}^{Q^{\ast}_{\mu}} \sigma^{Q^{\ast}_{\mu}-2}\rho Z_i^{2n+2}
}
{|\eta^{-1}\xi|^{\mu}} d\xi{d}\eta \nonumber\\
\lesssim&\left(\int_{\mathbb{H}^{n}}\bigg(\sigma^{Q^{\ast}_{\mu}-1}\rho\bigg)^{\frac{2Q}{2Q-\mu}} d\xi\right)^{\frac{2Q-\mu}{2Q}}\left(\int_{\mathbb{H}^{n}}\bigg(U_i^{Q^{\ast}_{\mu}-1}\sigma_{[i]}\bigg)^{\frac{2Q}{2Q-\mu}} 
d\xi\right)^{\frac{2Q-\mu}{2Q}} \nonumber\\
&+\sum\limits_{j\neq i}\int_{\mathbb{H}^{n}}\int_{\mathbb{H}^{n}}
\frac{U_j^{Q^{\ast}_{\mu}}\sigma^{Q^{\ast}_{\mu}-2}\rho Z_i^{2n+2}
}
{|\eta^{-1}\xi|^{\mu}} d\xi{d}\eta
 \nonumber\\ \lesssim& \varepsilon^{\frac{Q-2}{2}}\|\rho\|_{S^{1,2}(\mathbb{H}^{n})}+\sum\limits_{j\neq i}\int_{\mathbb{H}^{n}}U_j^{Q^*-Q^{\ast}_{\mu}}\sigma^{Q^{\ast}_{\mu}-2}\rho Z_i^{2n+2} d \xi  \nonumber\\
\lesssim&
\|\rho\|_{S^{1,2}(\mathbb{H}^{n})}^2
+o\Big(\varepsilon^{\frac{Q-2}{2}}\Big)+\int_{\mathbb{H}^{n}}\sigma^{Q^{\ast}-2}\rho Z_i^{2n+2} d \xi\overset{\eqref{main04}}= \|\rho\|_{S^{1,2}(\mathbb{H}^{n})}^2
+o\Big(\varepsilon^{\frac{Q-2}{2}}\Big),
\end{align}
and
\begin{align}\label{main45}
  |IV| \lesssim&
  \int_{\mathbb{H}^{n}}\int_{\mathbb{H}^{n}}\frac{U_i^{Q^{\ast}_{\mu}}
\sigma_{[i]}^{Q^{\ast}_{\mu}-2}\rho Z_i^{2n+2}}
{|\eta^{-1}\xi|^{\mu}}d\xi{d}\eta\lesssim
  \int_{\mathbb{H}^{n}}\int_{\mathbb{H}^{n}}\frac{U_i^{Q^{\ast}_{\mu}}U_i \rho
\sigma_{[i]}^{Q^{\ast}_{\mu}-2}}
{|\eta^{-1}\xi|^{\mu}}d\xi{d}\eta
\nonumber\\
\lesssim&\int_{\mathbb{H}^{n}}U_i^{Q^{\ast}-Q_\mu^*+1}\rho\sigma_{[i]}^{Q^{\ast}_{\mu}-2}\chi_{\{|\rho|\geq \sigma_{[i]}\}}d\xi+\int_{\mathbb{H}^{n}}U_i^{Q^{\ast}-Q_\mu^*+1}\rho\sigma_{[i]}^{Q^{\ast}_{\mu}-2}\chi_{\{|\rho|< \sigma_{[i]}\}}d\xi \nonumber\\
\lesssim&
\int_{\mathbb{H}^{n}}U_i^{Q^{\ast}-Q_\mu^*+1}\rho^{Q^{\ast}_{\mu}-1}d\xi+\int_{\mathbb{H}^{n}}U_i^{Q^{\ast}-Q_\mu^*+1}\sigma_{[i]}^{Q^{\ast}_{\mu}-1}
d\xi \nonumber\\
\lesssim&
\|\rho\|_{S^{1,2}(\mathbb{H}^{n})}^{Q^{\ast}_{\mu}-1}+o\Big(\varepsilon^{\frac{Q-2}{2}}\Big).
\end{align}
According to \eqref{main40}-\eqref{main45}, we deduce that
 \begin{equation}\label{main46}
   \left| \int_{\mathbb{H}^n} f Z_i^{2n+2} d\xi\right|
\lesssim \|\rho\|_{S^{1,2}(\mathbb{H}^{n})}^{\min\{2,Q_\mu^*-1\}}
     +\|h\|_{(S^{1,2}(\mathbb{H}^{n}))^{-1}}+o\Big(\varepsilon^{\frac{Q-2}{2}}\Big).
 \end{equation}

 Without loss of generality, we assume that $\lambda_1\geq \lambda_2 \geq\cdots \geq \lambda_m$. For a fixed $i_0$, define
 \begin{equation*}
   \varepsilon_{i_0}=\max\limits_{j\neq i_0}\{\varepsilon_{i_0j}\}.
 \end{equation*}
 Then, by \eqref{main01} and \eqref{main37}, setting $i=1$ in \eqref{main46}, we obtain
 \begin{equation}\label{main47}
   \varepsilon_1^{\frac{Q-2}{2}}
\lesssim \|\rho\|_{S^{1,2}(\mathbb{H}^{n})}^{\min\{2,Q_\mu^*-1\}}
     +\|h\|_{(S^{1,2}(\mathbb{H}^{n}))^{-1}}+o\Big(\varepsilon^{\frac{Q-2}{2}}\Big).
 \end{equation}
 Next, taking $i=2$ in \eqref{main46}, we deduce from \eqref{main00}, \eqref{main01}, \eqref{main37}, and \eqref{main47} that
 \begin{align*}
   \varepsilon_2^{\frac{Q-2}{2}}\approx  \sum_{i\geq3}
     \int U_2^{Q^*-2}U_i Z_2^{2n+2}d\xi  \lesssim & \left|\int U_2^{Q^*-2}U_1 Z_2^{2n+2}d\xi\right|+\|\rho\|_{S^{1,2}(\mathbb{H}^{n})}^{\min\{2,Q_\mu^*-1\}}
     +\|h\|_{(S^{1,2}(\mathbb{H}^{n}))^{-1}}+o\Big(\varepsilon^{\frac{Q-2}{2}}\Big) \\
     \lesssim &
     ~ \varepsilon_1^{\frac{Q-2}{2}} +\|\rho\|_{S^{1,2}(\mathbb{H}^{n})}^{\min\{2,Q_\mu^*-1\}}
     +\|h\|_{(S^{1,2}(\mathbb{H}^{n}))^{-1}}+o\Big(\varepsilon^{\frac{Q-2}{2}}\Big)\\
     \lesssim & \|\rho\|_{S^{1,2}(\mathbb{H}^{n})}^{\min\{2,Q_\mu^*-1\}}
     +\|h\|_{(S^{1,2}(\mathbb{H}^{n}))^{-1}}+o\Big(\varepsilon^{\frac{Q-2}{2}}\Big).
 \end{align*}
By induction, the desired result follows.
\end{proof}

\begin{lemma}\label{main5}
Let $Q\geq 4$, $0<\mu< Q$ with $0<\mu< 4$. Then there exists a constant $\hat{\delta}=\hat{\delta}(n,m)> 0$ such that for any $\delta\in (0,\hat{\delta})$, for any collection of ${\delta}$-weakly interacting bubbles $\big\{\mathfrak{g}_{i } U\big\}_{i=1}^m$, if $\rho$ satisfies 
\begin{equation*}
  L(\rho) = \varphi, \quad \langle \rho, Z_i^a \rangle_{S^{1,2}(\mathbb{H}^{n})}= 0 ,\quad\forall\, 1\leq i\leq m,\ 1\leq a\leq 2n+2,
\end{equation*}
then 
\begin{equation*}
 \|L(\rho) \|_{(S^{1,2}(\mathbb{H}^{n}))^{-1}} \gtrsim \|\rho\|_{_{S^{1,2}(\mathbb{H}^{n})}}.
\end{equation*}
\end{lemma}
\begin{proof}
Assume by contradiction that there exists a family of ${\delta_k}$-weakly interacting bubbles $\big\{\mathfrak{g}_{i }^{(k)} U\big\}_{i=1}^m$ such that
$\delta_k\rightarrow0$, and $\rho_k$ with $\|L(\rho_k) \|_{(S^{1,2}(\mathbb{H}^{n}))^{-1}}=o_k(1)$, $\|\rho_k\|_{_{S^{1,2}(\mathbb{H}^{n})}}=1$ satisfying
\begin{equation}\label{main51}
  -\Delta_{\mathbb{H} } \rho_k-Q^{\ast}_{\mu}\left(\int_{\mathbb{H}^{n}}\frac{|\sigma_k(\eta)|^{Q^{\ast}_{\mu}-1}\rho_k(\eta)}
{|\eta^{-1}\xi|^{\mu}}{d}\eta\right)|\sigma_k|^{Q^{\ast}_{\mu}-2}\sigma_k-(Q^{\ast}_{\mu}-1)
\left(\int_{\mathbb{H}^{n}}\frac{|\sigma_k(\eta)|^{Q^{\ast}_{\mu}}}{|\eta^{-1}\xi|^{\mu}}{d}\eta\right)
|\sigma_k|^{Q^{\ast}_{\mu}-2}\rho_k = \varphi_k,
\end{equation}
and $\langle \rho_k, Z_{i,k}^a \rangle_{S^{1,2}(\mathbb{H}^{n})}= 0$ for any $1\leq i\leq m $ and $ 1\leq a\leq 2n+2$,
where
\begin{equation*}
  \mathfrak{g}_i^{(k)}=\mathfrak{g}_{\lambda_i^{(k)},\xi_i^{(k)}},\quad \sigma_k=\sum\limits_{i=1}^m \mathfrak{g}_i^{(k)}U, \quad Z_{i,k}^a=\mathfrak{g}_i^{(k)} Z^a.
\end{equation*}

We claim that, up to a subsequence,
\begin{equation}\label{main50}
  \Big(\mathfrak{g}_1^{(k)}\Big)^{-1}\rho_k\rightharpoonup0 \qquad\mathrm{in}~S^{1,2}(\mathbb{H}^{n}).
\end{equation}
In fact, one can verify that
\begin{align*}
  -\Delta_{\mathbb{H} } \Big(\mathfrak{g}_1^{(k)}\Big)^{-1}\rho_k&-Q^{\ast}_{\mu}\left(\int_{\mathbb{H}^{n}}
  \frac{\left(U(\eta)+\sum\limits_{j\geq2}\Big(\mathfrak{g}_1^{(k)}\Big)^{-1} \mathfrak{g}_j^{(k)} U(\eta) \right)^{Q^{\ast}_{\mu}-1} \Big(\mathfrak{g}_1^{(k)}\Big)^{-1}\rho_k(\eta)}
{|\eta^{-1}\xi|^{\mu}}{d}\eta\right)\\
  &\quad \times\left(U+\sum\limits_{j\geq2}\Big(\mathfrak{g}_1^{(k)}\Big)^{-1} \mathfrak{g}_j^{(k)} U \right)^{Q^{\ast}_{\mu}-1}
\\
&-(Q^{\ast}_{\mu}-1)\left(\int_{\mathbb{H}^{n}}
  \frac{\left(U(\eta)+\sum\limits_{j\geq2}\Big(\mathfrak{g}_1^{(k)}\Big)^{-1} \mathfrak{g}_j^{(k)} U(\eta) \right)^{Q^{\ast}_{\mu}} 
  }
{|\eta^{-1}\xi|^{\mu}}{d}\eta\right)\\
&\quad \times \left(U(\eta)+\sum\limits_{j\geq2}\Big(\mathfrak{g}_1^{(k)}\Big)^{-1} \mathfrak{g}_j^{(k)} U \right)^{Q^{\ast}_{\mu}-2} \Big(\mathfrak{g}_1^{(k)}\Big)^{-1}\rho_k\\
=&\Big(\lambda_1^{(k)}\Big)^{-2}\Big(\mathfrak{g}_1^{(k)}\Big)^{-1}\varphi_k=o_k(1)\in (S^{1,2}(\mathbb{H}^{n}))^{-1},
\end{align*}
where we have used the fact that
\begin{equation*}
  \|v\|_{(S^{1,2}(\mathbb{H}^{n}))^{-1}}=\|\lambda^{-2}\mathfrak{g}^{-1}v\|_{(S^{1,2}(\mathbb{H}^{n}))^{-1}},\quad \forall \, v\in (S^{1,2}(\mathbb{H}^{n}))^{-1},\ \mathfrak{g}\in \mathcal{G}.
\end{equation*}
Since the bubbles $\big\{\big(\mathfrak{g}_1^{(k)}\big)^{-1} \mathfrak{g}_j^{(k)} U \big\}_{j=1}^{m}$ are also $\delta_k$-weakly interacting,
 without loss of generality, we assume that $\mathfrak{g}_1^{(k)}= \mathrm{id}$.
 Moreover, it follows from $\delta_k\rightarrow0$,  Remark \ref{iff} and Lemma \ref{gg} that $\mathfrak{g}_j^{(k)}\rightharpoonup0$ and $\big(\mathfrak{g}_j^{(k)}\big)^{-1}\rightharpoonup0$ in $S^{1,2}(\mathbb{H}^{n})$, thus in $L^{Q^*}(\mathbb{H}^{n})$ for any $j \geq 2$.
 Hence the equation becomes
 \begin{align*}
  -\Delta_{\mathbb{H} } \rho_k&-Q^{\ast}_{\mu}\left(\int_{\mathbb{H}^{n}}
  \frac{\bigg(U(\eta)+\sum\limits_{j\geq2} \mathfrak{g}_j^{(k)} U(\eta) \bigg)^{Q^{\ast}_{\mu}-1} \rho_k(\eta)}
{|\eta^{-1}\xi|^{\mu}}{d}\eta\right)\bigg(U+\sum\limits_{j\geq2} \mathfrak{g}_j^{(k)} U \bigg)^{Q^{\ast}_{\mu}-1}
\\
&-(Q^{\ast}_{\mu}-1)\left(\int_{\mathbb{H}^{n}}
  \frac{\bigg(U(\eta)+\sum\limits_{j\geq2} \mathfrak{g}_j^{(k)} U(\eta) \bigg)^{Q^{\ast}_{\mu}} 
  }
{|\eta^{-1}\xi|^{\mu}}{d}\eta\right)\bigg(U(\eta)+\sum\limits_{j\geq2} \mathfrak{g}_j^{(k)} U \bigg)^{Q^{\ast}_{\mu}-2} \rho_k
=o_k(1).
\end{align*}
Since $\|\rho_k\|_{_{S^{1,2}(\mathbb{H}^{n})}}=1$, up to a subsequence, we assume that
\begin{equation*}
  \rho_k\rightharpoonup \rho \qquad\mathrm{in}~S^{1,2}(\mathbb{H}^{n}).
\end{equation*}
Denote $\sigma_{*}=\sum\limits_{j\geq2} \mathfrak{g}_j^{(k)} U$.
For any test function $\phi\in C_c^\infty$, by \eqref{eq:HLSH} and Lemma \ref{gs}, we have
\begin{align*}
  &\int_{\mathbb{H}^{n}}\int_{\mathbb{H}^{n}}
  \frac{\bigg(U+\sum\limits_{j\geq2} \mathfrak{g}_j^{(k)} U \bigg)^{Q^{\ast}_{\mu}-1} \rho_k      \bigg(U+\sum\limits_{j\geq2} \mathfrak{g}_j^{(k)} U \bigg)^{Q^{\ast}_{\mu}-1} \phi-U^{Q^{\ast}_{\mu}-1} \rho_k U^{Q^{\ast}_{\mu}-1} \phi}
{|\eta^{-1}\xi|^{\mu}}d\xi{d}\eta\\
=&\int_{\mathbb{H}^{n}}\int_{\mathbb{H}^{n}}
  \frac{\Big[(U+\sigma_* )^{Q^{\ast}_{\mu}-1}-U^{Q^{\ast}_{\mu}-1}\Big] \rho_k      (U+\sigma_*)^{Q^{\ast}_{\mu}-1} \phi
  }
{|\eta^{-1}\xi|^{\mu}}d\xi{d}\eta\\
&+\int_{\mathbb{H}^{n}}\int_{\mathbb{H}^{n}}
  \frac{U^{Q^{\ast}_{\mu}-1} \rho_k      \Big[(U+\sigma_* )^{Q^{\ast}_{\mu}-1}-U^{Q^{\ast}_{\mu}-1}\Big] \phi}
{|\eta^{-1}\xi|^{\mu}}d\xi{d}\eta\\
\lesssim &
 \int_{\mathbb{H}^{n}}\int_{\mathbb{H}^{n}}
  \frac{\Big[U^{Q^{\ast}_{\mu}-2}\sigma_* +\sigma_*^{Q^{\ast}_{\mu}-1}\Big] \rho_k      (U+\sigma_*)^{Q^{\ast}_{\mu}-1} \phi
  }
{|\eta^{-1}\xi|^{\mu}}d\xi{d}\eta \\
&+\int_{\mathbb{H}^{n}}\int_{\mathbb{H}^{n}}
  \frac{U^{Q^{\ast}_{\mu}-1} \rho_k      \Big[U^{Q^{\ast}_{\mu}-2}\sigma_* +\sigma_*^{Q^{\ast}_{\mu}-1}\Big] \phi}
{|\eta^{-1}\xi|^{\mu}}d\xi{d}\eta\\
\lesssim &
\begin{cases}
\displaystyle \sum\limits_{j\geq2}\int_{\mathbb{H}^{n}}\int_{\mathbb{H}^{n}}
  \frac{U^{Q^{\ast}_{\mu}-2}\Big(\mathfrak{g}_j^{(k)} U\Big)  \rho_k      U^{Q^{\ast}_{\mu}-1} \phi
  }
{|\eta^{-1}\xi|^{\mu}}d\xi{d}\eta,\qquad &\mathrm{if}~Q_\mu^*\geq3,\\
\displaystyle \sum\limits_{j\geq2}\int_{\mathbb{H}^{n}}\int_{\mathbb{H}^{n}}
  \frac{U^{Q^{\ast}_{\mu}-2}\Big(\mathfrak{g}_j^{(k)} U\Big) \rho_k      \Big(\mathfrak{g}_j^{(k)} U\Big)^{Q^{\ast}_{\mu}-1} \phi
  }
{|\eta^{-1}\xi|^{\mu}}d\xi{d}\eta,\qquad &\mathrm{if}~2< Q_\mu^*<3,
\end{cases}
\\
&+\sum\limits_{j\geq2}\int_{\mathbb{H}^{n}}\int_{\mathbb{H}^{n}}
  \frac{\Big(\mathfrak{g}_j^{(k)} U\Big)^{Q^{\ast}_{\mu}-1} \rho_k      \Big(\mathfrak{g}_j^{(k)} U\Big)^{Q^{\ast}_{\mu}-1} \phi
  }
{|\eta^{-1}\xi|^{\mu}}d\xi{d}\eta
\\
&+\sum\limits_{j\geq2}\int_{\mathbb{H}^{n}}\int_{\mathbb{H}^{n}}
  \frac{U^{Q^{\ast}_{\mu}-1} \rho_k      \bigg[U^{Q^{\ast}_{\mu}-2}\mathfrak{g}_j^{(k)} U +\Big(\mathfrak{g}_j^{(k)} U\Big)^{Q^{\ast}_{\mu}-1}\bigg] \phi}
{|\eta^{-1}\xi|^{\mu}}d\xi{d}\eta\\
\lesssim &
 \begin{cases}
 \sum\limits_{j\geq2} \left(\underbrace{\int_{\mathbb{H}^{n}}\Big(U^{Q^{\ast}_{\mu}-2} \mathfrak{g}_j^{(k)} U  \Big)^{\frac{2Q}{Q+2-\mu}}d\xi}_{:=I_j ^{(k)}}\right )^{\frac{Q+2-\mu}{2Q}}\qquad &\mathrm{if}~ Q_\mu^*\geq3,\\
 \sum\limits_{j\geq2} \left(\underbrace{\int_{\mathbb{H}^{n}}\Big(\mathfrak{g}_j^{(k)} U\Big)   ^{\frac{2Q(Q^{\ast}_{\mu}-1)}{2Q-\mu}}\phi^{\frac{2Q}{2Q-\mu}}d\xi}_{:=II_j ^{(k)}}\right )^{\frac{2Q-\mu}{2Q}}\qquad &\mathrm{if}~2< Q_\mu^*<3,
 \end{cases}
 \\
 &+\sum\limits_{j\geq2} \left(\underbrace{\int_{\mathbb{H}^{n}}\Big(\mathfrak{g}_j^{(k)} U\Big)   ^{\frac{2Q(Q^{\ast}_{\mu}-1)}{2Q-\mu}}\phi^{\frac{2Q}{2Q-\mu}}d\xi}_{:=III_j ^{(k)}}\right )^{\frac{2Q-\mu}{2Q}}+
 \sum\limits_{j\geq2} \left(\underbrace{\int_{\mathbb{H}^{n}}\Big(U^{Q^{\ast}_{\mu}-2} \phi\mathfrak{g}_j^{(k)} U  \Big)^{\frac{2Q}{2Q-\mu}}d\xi}_{:=IV_j ^{(k)}}\right )^{\frac{2Q-\mu}{2Q}}.
\end{align*}
Since $\big(\mathfrak{g}_j^{(k)} U  \big)^{\frac{2Q}{Q+2-\mu}}\rightharpoonup 0$ in $L^{\frac{Q+2-\mu}{Q-2}}(\mathbb{H}^{n})$ and $U^{\frac{2Q(Q_\mu^*-2)}{Q+2-\mu}}$ is bounded in $L^{\frac{Q+2-\mu}{4-\mu}}(\mathbb{H}^{n})$, by the definition of weak convergence, we obtain
\begin{equation*}
  I_j^{(k)}\rightarrow0,\qquad \mathrm{as} ~ k\rightarrow+\infty.
\end{equation*}
Similarly, we can prove that
\begin{equation*}
  II_j^{(k)},III_j^{(k)},IV_j^{(k)}\rightarrow0,\qquad \mathrm{as} ~ k\rightarrow+\infty.
\end{equation*}
Analogously, by $\mu<4$, we get 
\begin{align}\label{main52}
  &\int_{\mathbb{H}^{n}}\int_{\mathbb{H}^{n}}
  \frac{\bigg(U+\sum\limits_{j\geq2} \mathfrak{g}_j^{(k)} U \bigg)^{Q^{\ast}_{\mu}}       \bigg(U+\sum\limits_{j\geq2} \mathfrak{g}_j^{(k)} U \bigg)^{Q^{\ast}_{\mu}-2}\rho_k \phi-U^{Q^{\ast}_{\mu}}  U^{Q^{\ast}_{\mu}-2}\rho_k \phi}
{|\eta^{-1}\xi|^{\mu}}d\xi{d}\eta \nonumber\\
=&\int_{\mathbb{H}^{n}}\int_{\mathbb{H}^{n}}
  \frac{\Big[(U+\sigma_* )^{Q^{\ast}_{\mu}}-U^{Q^{\ast}_{\mu}}\Big]       (U+\sigma_*)^{Q^{\ast}_{\mu}-2}\rho_k \phi
  }
{|\eta^{-1}\xi|^{\mu}}d\xi{d}\eta +\int_{\mathbb{H}^{n}}\int_{\mathbb{H}^{n}}
  \frac{U^{Q^{\ast}_{\mu}}       \Big[(U+\sigma_* )^{Q^{\ast}_{\mu}-2}-U^{Q^{\ast}_{\mu}-2}\Big]\rho_k \phi}
{|\eta^{-1}\xi|^{\mu}}d\xi{d}\eta \nonumber\\
\lesssim &
 \int_{\mathbb{H}^{n}}\int_{\mathbb{H}^{n}}
  \frac{\Big[U^{Q^{\ast}_{\mu}-1}\sigma_* +\sigma_*^{Q^{\ast}_{\mu}}\Big]     (U+\sigma_*)^{Q^{\ast}_{\mu}-2} \rho_k   \phi
  }
{|\eta^{-1}\xi|^{\mu}}d\xi{d}\eta
\nonumber \\&+\begin{cases}
\displaystyle\int_{\mathbb{H}^{n}}\int_{\mathbb{H}^{n}}
  \frac{U^{Q^{\ast}_{\mu}}      \Big[U^{Q^{\ast}_{\mu}-3}\sigma_*+\sigma_*^{Q^{\ast}_{\mu}-2}\Big]\rho_k \phi}
{|\eta^{-1}\xi|^{\mu}}d\xi{d}\eta, \qquad &\mathrm{if}~Q_\mu^*\geq3,\nonumber \\
\displaystyle\int_{\mathbb{H}^{n}}\int_{\mathbb{H}^{n}}
  \frac{U^{Q^{\ast}_{\mu}}      \sigma_*^{Q^{\ast}_{\mu}-2}\rho_k \phi}
{|\eta^{-1}\xi|^{\mu}}d\xi{d}\eta, \qquad &\mathrm{if}~2< Q_\mu^*<3,
\end{cases}
\nonumber \\
\lesssim &
 \sum\limits_{j\geq2}\int_{\mathbb{H}^{n}}\int_{\mathbb{H}^{n}}
  \frac{U^{Q^{\ast}_{\mu}-1}\Big(\mathfrak{g}_j^{(k)} U\Big)   U^{Q^{\ast}_{\mu}-2} \rho_k \phi
  }
{|\eta^{-1}\xi|^{\mu}}d\xi{d}\eta+\sum\limits_{j\geq2}\int_{\mathbb{H}^{n}}\int_{\mathbb{H}^{n}}
  \frac{\Big(\mathfrak{g}_j^{(k)} U\Big)^{Q^{\ast}_{\mu}}     \Big(\mathfrak{g}_j^{(k)} U\Big)^{Q^{\ast}_{\mu}-2} \rho_k \phi
  }
{|\eta^{-1}\xi|^{\mu}}d\xi{d}\eta \nonumber\\
&+\begin{cases}
\displaystyle\sum\limits_{j\geq2}\int_{\mathbb{H}^{n}}\int_{\mathbb{H}^{n}}
  \frac{U^{Q^{\ast}_{\mu}}    \bigg[U^{Q^{\ast}_{\mu}-3}   \mathfrak{g}_j^{(k)} U+\Big(\mathfrak{g}_j^{(k)} U\Big)^{Q^{\ast}_{\mu}-2}\bigg]\rho_k \phi}
{|\eta^{-1}\xi|^{\mu}}d\xi{d}\eta, \qquad &\mathrm{if}~Q_\mu^*\geq3, \nonumber\\
\displaystyle\sum\limits_{j\geq2}\int_{\mathbb{H}^{n}}\int_{\mathbb{H}^{n}}
  \frac{U^{Q^{\ast}_{\mu}}       \Big(\mathfrak{g}_j^{(k)} U\Big)^{Q^{\ast}_{\mu}-2}\rho_k \phi}
{|\eta^{-1}\xi|^{\mu}}d\xi{d}\eta, \qquad &\mathrm{if}~2< Q_\mu^*<3,
\end{cases}
\\
 \lesssim &
 \sum\limits_{j\geq2} \left(\int_{\mathbb{H}^{n}}\Big(U^{Q^{\ast}_{\mu}-1} \mathfrak{g}_j^{(k)} U  \Big)^{\frac{2Q}{2Q-\mu}} d\xi\right )^{\frac{2Q-\mu}{2Q}}
 +\sum\limits_{j\geq2} \left(\int_{\mathbb{H}^{n}}\Big( \mathfrak{g}_j^{(k)} U  \Big)^{\frac{2Q(Q_\mu^*-2)}{Q+2-\mu}}\phi^{\frac{2Q}{Q+2-\mu}} d\xi\right )^{\frac{Q+2-\mu}{2Q}}
\nonumber \\
 &+\begin{cases}
\displaystyle\sum\limits_{j\geq2} \left(\int_{\mathbb{H}^{n}}\Big(  U^{Q^{\ast}_{\mu}-3}  \phi \mathfrak{g}_j^{(k)} U   \Big)^{\frac{2Q(Q_\mu^*-2)}{Q+2-\mu}} d\xi\right )^{\frac{Q+2-\mu}{2Q}}\\
\displaystyle+\sum\limits_{j\geq2} \left(\int_{\mathbb{H}^{n}}\Big( \mathfrak{g}_j^{(k)} U  \Big)^{\frac{2Q(Q_\mu^*-2)}{Q+2-\mu}}\phi^{\frac{2Q}{Q+2-\mu}} d\xi\right )^{\frac{Q+2-\mu}{2Q}}, \qquad &\mathrm{if}~Q_\mu^*\geq3,\nonumber\\
\displaystyle\sum\limits_{j\geq2} \left(\int_{\mathbb{H}^{n}}\Big( \mathfrak{g}_j^{(k)} U  \Big)^{\frac{2Q(Q_\mu^*-2)}{Q+2-\mu}}\phi^{\frac{2Q}{Q+2-\mu}} d\xi\right )^{\frac{Q+2-\mu}{2Q}}, \qquad &\mathrm{if}~2< Q_\mu^*<3,
\end{cases}
 \\
 &\longrightarrow  0,\qquad \mathrm{as} ~ k\rightarrow+\infty.
\end{align}
Therefore, by letting $k\rightarrow +\infty$, we obtain 
 \begin{align*}
  -\Delta_{\mathbb{H} } \rho&-Q^{\ast}_{\mu}\left(\int_{\mathbb{H}^{n}}\frac{|U(\eta)|^{Q^{\ast}_{\mu}-1}\rho(\eta)}
{|\eta^{-1}\xi|^{\mu}}{d}\eta\right)|U|^{Q^{\ast}_{\mu}-2}U+(Q^{\ast}_{\mu}-1)
\left(\int_{\mathbb{H}^{n}}\frac{|U(\eta)|^{Q^{\ast}_{\mu}}}{|\eta^{-1}\xi|^{\mu}}{d}\eta\right)
|U|^{Q^{\ast}_{\mu}-2}\rho
=0.
\end{align*}
On the other hand, it follows from $\langle \mathfrak{g}v, w \rangle_{S^{1,2}(\mathbb{H}^{n})}=\langle v, \mathfrak{g}^{-1}w \rangle_{S^{1,2}(\mathbb{H}^{n})}$ for any $v,w\in S^{1,2}(\mathbb{H}^{n})$ that
\begin{equation*}
  \langle \rho, Z^a \rangle_{S^{1,2}(\mathbb{H}^{n})}=\lim\limits_{k\rightarrow+\infty}\Big\langle \Big(\mathfrak{g}_1^{(k)}\Big)^{-1}\rho_k, Z^a \Big\rangle_{S^{1,2}(\mathbb{H}^{n})}=\lim\limits_{k\rightarrow+\infty}\langle \rho_k, Z_{1,k}^a \rangle_{S^{1,2}(\mathbb{H}^{n})}= 0.
\end{equation*}
This with Lemma \ref{nondege} yields that $\rho=0$. Hence, we prove the claim \eqref{main50}.

A similar discussion yields that, up to a subsequence, for each $1\leq i  \leq m$, $\big(\mathfrak{g}_i^{(k)}\big)^{-1}\rho_k\rightharpoonup0 $ in $S^{1,2}(\mathbb{H}^{n})$.
Now, testing the equation \eqref{main51} by $\rho_k$,
since
\begin{equation*}
  \int_{\mathbb{H}^{n}}\int_{\mathbb{H}^{n}}\frac{\Big(\mathfrak{g}_i^{(k)} U\Big) ^{\alpha_1}v^{\alpha_2}
\Big(\mathfrak{g}_i^{(k)} U\Big) ^{\beta_1}v^{\beta_2}}{|\eta^{-1}\xi|^{\mu}}{d}\xi d\eta=\int_{\mathbb{H}^{n}}\int_{\mathbb{H}^{n}}\frac{U^{\alpha_1}\Big[\Big(\mathfrak{g}_i^{(k)}\Big)^{-1} v\Big] ^{\alpha_2}
U^{\beta_1}\Big[\Big(\mathfrak{g}_i^{(k)}\Big)^{-1} v\Big] ^{\beta_2}}{|\eta^{-1}\xi|^{\mu}}{d}\xi d\eta
\end{equation*}
for any $v\in S^{1,2}(\mathbb{H}^{n})$, and any $\alpha_1,\alpha_2,\beta_1,\beta_2\geq0$ with $\alpha_1+\alpha_2=Q_\mu^*$, $\beta_1+\beta_2=Q_\mu^*$,
we have
\begin{align}\label{main53}
1=\|\rho_k\|_{_{S^{1,2}(\mathbb{H}^{n})}}^2 \lesssim&   \int_{\mathbb{H}^{n}}\int_{\mathbb{H}^{n}}
\frac{\sigma_k^{Q^{\ast}_{\mu}-1}\rho_k \sigma_k^{Q^{\ast}_{\mu}-1} \rho_k}
{|\eta^{-1}\xi|^{\mu}} d\xi{d}\eta+
\int_{\mathbb{H}^{n}}\int_{\mathbb{H}^{n}}\frac{\sigma_k^{Q^{\ast}_{\mu}}
\sigma_k^{Q^{\ast}_{\mu}-2}\rho_k^2}{|\eta^{-1}\xi|^{\mu}}{d}\xi d\eta+o_k(1) \nonumber\\
\lesssim&  \sum\limits_{i,j=1}^m \int_{\mathbb{H}^{n}}\int_{\mathbb{H}^{n}}
\frac{\Big(\mathfrak{g}_i^{(k)} U\Big) ^{Q^{\ast}_{\mu}-1}\rho_k \Big(\mathfrak{g}_j^{(k)} U\Big) ^{Q^{\ast}_{\mu}-1} \rho_k}
{|\eta^{-1}\xi|^{\mu}} d\xi{d}\eta  \nonumber\\
&+\sum\limits_{i,j=1}^m
\int_{\mathbb{H}^{n}}\int_{\mathbb{H}^{n}}\frac{\Big(\mathfrak{g}_i^{(k)} U\Big) ^{Q^{\ast}_{\mu}}
\Big(\mathfrak{g}_j^{(k)} U\Big) ^{Q^{\ast}_{\mu}-2}\rho_k^2}{|\eta^{-1}\xi|^{\mu}}{d}\xi d\eta+o_k(1)  \nonumber\\
\lesssim&\sum\limits_{i=1}^m \int_{\mathbb{H}^{n}}\int_{\mathbb{H}^{n}}
\frac{U ^{Q^{\ast}_{\mu}-1}\Big[\Big(\mathfrak{g}_i^{(k)}\Big)^{-1}\rho_k\Big] U ^{Q^{\ast}_{\mu}-1}\Big[\Big(\mathfrak{g}_i^{(k)}\Big)^{-1}\rho_k\Big]}
{|\eta^{-1}\xi|^{\mu}} d\xi{d}\eta  \nonumber\\
&+\sum\limits_{i=1}^m
\int_{\mathbb{H}^{n}}\int_{\mathbb{H}^{n}}\frac{U ^{Q^{\ast}_{\mu}}
U^{Q^{\ast}_{\mu}-2}\Big[\Big(\mathfrak{g}_i^{(k)}\Big)^{-1}\rho_k\Big]^2}{|\eta^{-1}\xi|^{\mu}}{d}\xi d\eta+o_k(1)  \nonumber\\
\lesssim& \sum\limits_{i=1}^m \left(\int_{\mathbb{H}^{n}}
U^{\frac{2Q(Q^{\ast}_{\mu}-1)}{2Q-\mu}} \Big[\Big(\mathfrak{g}_i^{(k)}\Big)^{-1}\rho_k\Big]^{\frac{2Q}{2Q-\mu}} d\xi\right )^{\frac{2Q-\mu}{Q}}  \nonumber\\
&+\sum\limits_{i=1}^m \left(\int_{\mathbb{H}^{n}}
U^{\frac{2Q(Q^{\ast}_{\mu}-2)}{2Q-\mu}} \Big[\Big(\mathfrak{g}_i^{(k)}\Big)^{-1}\rho_k\Big]^{\frac{4Q}{2Q-\mu}} d\xi\right )^{\frac{2Q-\mu}{2Q}}+o_k(1)=o_k(1),
\end{align}
which is a contradiction.
Thus we finish the proof of Lemma \ref{main5}.
\end{proof}

We are now ready to prove Theorem \ref{main thm}.

  \begin{proof}
[Proof of Theorem \ref{main thm}]
When $Q=4$ and $\mu\in (2,4)$,  Lemma \ref{main4} yields that
    \begin{equation}\label{maineq1}
     \varepsilon \lesssim  \|\rho\|_{S^{1,2}(\mathbb{H}^{n})}^{3-\frac{\mu}{2}}+\|h\|_{(S^{1,2}(\mathbb{H}^{n}))^{-1}}.
   \end{equation}
   By Lemma \ref{main5}, we obtain
   \begin{equation*}
   \begin{aligned}
   \|\rho\|_{S^{1,2}(\mathbb{H}^{n})}\lesssim \|f+h+N(\rho)\|_{(S^{1,2}(\mathbb{H}^{n}))^{-1}}
   \lesssim \|f\|_{(S^{1,2}(\mathbb{H}^{n}))^{-1}}+\|h\|_{(S^{1,2}(\mathbb{H}^{n}))^{-1}}+
   \|N(\rho)\|_{(S^{1,2}(\mathbb{H}^{n}))^{-1}}.
   \end{aligned}
   \end{equation*}
  It follows from Lemmas \ref{main1}-\ref{main2} and $\mu<4$ that
  \begin{equation}\label{maineq2}
 \|\rho\|_{S^{1,2}(\mathbb{H}^{n})}  \lesssim \|f\|_{(S^{1,2}(\mathbb{H}^{n}))^{-1}}+\|h\|_{(S^{1,2}(\mathbb{H}^{n}))^{-1}} \lesssim  \varepsilon+\|h\|_{(S^{1,2}(\mathbb{H}^{n}))^{-1}}.
  \end{equation}
  Thus, by \eqref{maineq1}-\eqref{maineq2} and $\mu<4$, we deduce that
  \begin{equation*}
    \varepsilon \lesssim \big(\varepsilon+\|h\|_{(S^{1,2}(\mathbb{H}^{n}))^{-1}}\big)^{3-\frac{\mu}{2}}+
    \|h\|_{(S^{1,2}(\mathbb{H}^{n}))^{-1}} \lesssim \|h\|_{(S^{1,2}(\mathbb{H}^{n}))^{-1}}.
  \end{equation*}
Therefore,
  \begin{equation*}
    \|\rho\|_{S^{1,2}(\mathbb{H}^{n})}  \lesssim \|h\|_{(S^{1,2}(\mathbb{H}^{n}))^{-1}},
  \end{equation*}
   and the proof is completed.
\end{proof}

\subsection*{Acknowledgments}
The research has been supported by Natural Science Foundation of Chongqing, China (CSTB2024NSCQ-LZX0038), and Chongqing Graduate Student Research Innovation Project
(CYB25100).

\subsection*{Conflict of interest}

 No potential conflict of interest was reported by the author(s).

\subsection*{Data Availability}

Date sharing is not applicable to this article as no new data were created analyzed in this study.

\end{document}